\documentclass[12pt,reqno]{amsart}

\usepackage{amsfonts}
\usepackage{amscd}
\usepackage{amssymb}
\usepackage{amsfonts}
\usepackage{amscd}
\usepackage{amsmath} 
\usepackage{mathabx}
\usepackage{mathrsfs}
\usepackage{dsfont}
\usepackage{enumerate}
\usepackage{xcolor}
\usepackage{float} 
\usepackage[pdftex]{graphicx}
\allowdisplaybreaks
\usepackage{url}
\usepackage[hidelinks]{hyperref} 
 \hypersetup{
     colorlinks,
     linkcolor={red!70!black},
     citecolor={blue!80!black},
     urlcolor={blue!60!black}
 }

\date{\today}

\newtheorem*{theorem*}{Theorem}
\newtheorem{theorem}{Theorem}[section]
\newtheorem{corollary}[theorem]{\bf{Corollary}}

\theoremstyle{definition}

\theoremstyle{remark}
\newtheorem{remark}[theorem]{\bf{Remark}}

\numberwithin{equation}{section}

\newcommand{\beas}{\begin{eqnarray*}}
\newcommand{\eeas}{\end{eqnarray*}}
\newcommand{\bes} {\begin{equation*}}
\newcommand{\ees} {\end{equation*}}
\newcommand{\be} {\begin{equation}}
\newcommand{\ee} {\end{equation}}
\newcommand{\bea} {\begin{eqnarray}}
\newcommand{\eea} {\end{eqnarray}}

\newcommand{\R}{\mathbb R}
\newcommand{\C}{\mathbb C}
\newcommand{\Z}{\mathbb Z}%
\newcommand{\N}{\mathbb N}

\newcommand{\X}{\mathbb{X}}
\newcommand{\what}{\widehat}

\renewcommand{\Im}{\text{Im}}
\renewcommand{\Re}{\text{Re}}

\renewcommand{\Re}{\operatorname{Re}}
\renewcommand{\Im}{\operatorname{Im}}
 
\usepackage{color}

\title[Classical inequalities and Fourier multipliers on $\mathrm{SL}(2,\R)$]{Classical inequalities for all Fourier matrix coefficients of $\mathrm{SL}(2,\mathbb{R})$ and their applications}

\author[Kumar, Rana, Ruzhansky]{Vishvesh Kumar, Tapendu Rana, Michael Ruzhansky}	
\address{Vishvesh Kumar  \endgraf Department of Mathematics: Analysis, Logic and Discrete Mathematics,	\endgraf Ghent University, 	\endgraf Krijgslaan 281, Building S8, B 9000 Ghent, Belgium.} \email{vishveshmishra@gmail.com, Vishvesh.Kumar@ugent.be}
\address{Tapendu Rana  \endgraf Department of Mathematics: Analysis, Logic and Discrete Mathematics,	\endgraf Ghent University, 	\endgraf Krijgslaan 281, Building S8, B 9000 Ghent, Belgium.} \email{tapendurana@gmail.com, tapendu.rana@ugent.be}
\address{Michael Ruzhansky \endgraf Department of Mathematics: Analysis, Logic and Discrete Mathematics,	\endgraf Ghent University, 	\endgraf Krijgslaan 281, Building S8, B 9000 Ghent, Belgium. 
\endgraf and
\endgraf School of Mathematical Sciences
\endgraf Queen Mary University of London
\endgraf United Kingdom}
\email{michael.ruzhansky@ugent.be} 

\date{}
\keywords{Fourier multipliers on groups;  Hausdorff-Young inequality; Paley inequality;  spherical functions;   $K$-type functions on $\mathrm{SL}(2,\R)$; heat equation}
\subjclass[2020]{Primary: 43A22, 22E30;   Secondary: 43A80, 35K05}

\begin{document}
\maketitle
\begin{abstract}
In this article, we establish three fundamental Fourier inequalities: the Hausdorff-Young inequality, the Paley inequality, and the Hausdorff-Young-Paley inequality for $(l, n)$-type functions on $\mathrm{SL}(2,\R)$. Utilizing these inequalities, we demonstrate the $L^p$-$L^q$ boundedness of $(l, n)$-type Fourier multipliers on $\mathrm{SL}(2,\R)$. Furthermore, we explore applications related to the $L^p$-$L^q$ estimates of the heat kernel of the Casimir element on $\mathrm{SL}(2,\R)$ and address the global well-posedness of certain parabolic and hyperbolic nonlinear equations.
\end{abstract}
\tableofcontents
\addtocontents{toc}{\setcounter{tocdepth}{2}}
\section{Introduction}

In the realm of harmonic analysis, it is essential to understand how integral operators behave across different function spaces. Multiplier operators, in particular, are powerful tools in this context, finding applications from Fourier analysis to partial differential equations. In their seminal papers, Mikhlin \cite{Mih56}  and  H\"ormander \cite{Hormander_Acta_60} initiated the study of the boundedness of translation-invariant operators in the Euclidean setting. These translation-invariant operators can be characterized using the classical Euclidean Fourier transform on $\mathbb{R}^d$, hence they are also known as Fourier multipliers. The Fourier multiplier operator associated with a symbol $m(\xi)$ is defined as
\begin{equation}\label{defn_T_M_Euclid}
T_m f(x)=\int_{\R^d} e^{2\pi i x \cdot \xi} m(\xi)\mathcal{F} f(\xi)\,d\xi,
\end{equation}
where $\mathcal{F}f $ denotes the Euclidean Fourier transform given by
\begin{equation*}
 \mathcal{F} f(\xi)= \int_{\R^d} f(x)  e^{-2\pi i x \cdot \xi} dx.
\end{equation*}
Mikhlin \cite{Mih56} and H\"{o}rmander \cite{Hormander_Acta_60} established the $L^p$-boundedness and $L^p$-$L^q$ boundedness of Fourier multipliers on $\mathbb{R}^d$. Consequently, the study of multiplier operators on $\mathbb{R}^d$, as well as in various other spaces, flourished, becoming an important and active research area in modern harmonic analysis and partial differential equations. In particular, the $L^p$-boundedness of Fourier multipliers has been investigated by several researchers in different settings;  {we cite  \cite{Hormander_Acta_60,Taylor,Meda_Stefano_2019,Meda_Wrobel_2021,MR3369343,BT} among others to mention a few.}

Hörmander observed the striking aspect of studying $L^p$-$L^q$ estimates of multiplier operators for the range $1 < p \leq 2 \leq q < \infty$ is that they do not rely on the regularity of the symbol, unlike the results concerning $L^p$-$L^p$ multipliers, where the symbol's smoothness is crucial. Furthermore, he established the following result concerning the $L^{p}$-$L^q$ boundedness problem of multiplier operator  $T_m$.
\begin{theorem}[{{\cite[Theorem 1.11]{Hormander_Acta_60}}}] \label{thm_hor_Lp,Lq}
    Let $ 1<p\leq 2\leq q<\infty$. Assume that there exists a constant $C>0$ such that 
    \begin{equation}
        \left|\left\{ \xi \in \R^d : |m(\xi)|>\alpha \right\} \right|^{\left(\frac{1}{p}-\frac{1}{q} \right)} \leq {C}/{\alpha},
    \end{equation}
    for all $\alpha >0$. Then the associated Fourier multiplier operator $T_m$ is bounded from $L^p(\R^d)$ to $L^q(\R^d)$.    
\end{theorem}
Recently, researchers have turned their attention to establishing the boundedness of $L^p$-$L^q$ multipliers in several different settings; we refer to \cite{Stanton_and_Tomas,CGM93,RR19,AMM18, RRN19,KR22, KR23, KR23II, CK23,CK22i,ruzhansky2024h,cardona2023p,MR4575780} and references therein.  In this article, we aim to explore the $L^p$-$L^q$ boundedness properties of multipliers on all Fourier matrix coefficients on $\mathrm{SL}(2,\mathbb{R})$, and their application to nonlinear equations.

The study of boundedness properties of multiplier operators has primarily focused on the doubling setting, relying on suitable covering lemmas. However, when dealing with spaces characterized by exponential volume growth, such as noncompact connected semisimple Lie groups or noncompact type symmetric spaces, the absence of analogs for the Calderon-Zygmund decomposition or useful covering lemmas becomes apparent. Despite these challenges, Clerc and Stein addressed the multiplier problem in their seminal work \cite{Clerc_Stein} on general noncompact type symmetric spaces.

Subsequently, the endeavor to extend the classical Hörmander-Mikhlin multiplier theorem to symmetric spaces of noncompact type has attracted the attention of several authors \cite{Stanton_and_Tomas,Anker-Lohoue,Taylor,Anker,GMM,Meda_Vallarino_10,Meda_Wrobel_2021}. In 1990, Anker, in his remarkable work \cite{Anker}, improved and generalized the previous results of Clerc and Stein \cite{Clerc_Stein}, Stanton and Tomas \cite{Stanton_and_Tomas}, Anker and Lohoue \cite{Anker-Lohoue}, and Taylor \cite{Taylor} by proving an analogue of the Hörmander-Mikhlin multiplier theorem on noncompact type symmetric spaces $\X$ of arbitrary rank. Later, Ionescu \cite{Ionescu_2002,Ionescu_2003} made further improvements to the theorem by replacing the continuity requirement of the multiplier $m$ on the boundary with a condition related to a possible singularity on the boundary. These results of Ionescu are currently the best-known sufficient conditions of the Hörmander-Mikhlin type on $\X$. Recently, Wróbel \cite{wrobel} presented a multiplier theorem for rank one symmetric spaces, which improves upon the results of both \cite{Stanton_and_Tomas} and \cite{Ionescu_2002}.

Motivated by the works of Anker \cite{Anker} and Ionescu \cite{Ionescu_2002,Ionescu_2003}, Ricci and Wr\'{o}bel \cite{Ricci} studied the $L^p$-boundedness ($1<p<\infty$) of multiplier operators on the radial sections of line bundles over the Poincaré upper half plane $\mathrm{SL}(2,\R)/\mathrm{SO(2)}$. More precisely, they focused on the group $G=\mathrm{SL}(2,\R)$, and instead of considering $K:=\mathrm{SO(2) }$-biinvariant functions, they explored the problem of the $L^p$-boundedness for the multiplier operator defined on functions on $G$ satisfying the property:
\begin{align}\label{n_ntype_int}
f(k_\theta g k_\vartheta) = e^{in(\theta+\vartheta)} f(g),
\end{align}
 for all $g \in G$, $k_\theta, k_{\vartheta}\in K$, where $n\in \Z$ is fixed and \begin{align*}
  k_\theta := \begin{pmatrix}
				\cos \theta  & \sin \theta   \\
				-\sin \theta & \cos \theta  
			\end{pmatrix}.
   \end{align*} 
   The functions satisfying \eqref{n_ntype_int} are called $(n,n)$ type functions on $G$. In the special case where $n=0$, their result \cite[Theorem 5.3]{Ricci} aligns with the multiplier theorems on symmetric spaces obtained by Anker \cite{Anker} and Stanton and Tomas \cite{Stanton_and_Tomas} in the $K$-biinvariant case. Moreover, the authors \cite{Ricci} extended the result of Clerc and Stein \cite[Theorem 1]{Clerc_Stein} in this setting by providing a necessary condition on the multiplier $m$ for the associated multiplier operator to be $L^p$-bounded for $p\in (1,\infty)\setminus \{2\}$. Subsequently, the second and third authors extended the results on multiplier operators from \cite{Ricci} to pseudo-differential operators in \cite{rana_23}.

    Inspired by the results above on $L^p$-boundedness of multipliers for $(n,n)$ type functions on $G$, our interest naturally extends to investigating the $L^p$-$L^q$ estimate of the multiplier operator for more general types of functions on $G$. Specifically, we have established the following boundedness result for the multiplier operator defined on functions on $G$ that satisfy the property
\begin{align}\label{l_ntype_int}
f(k_\theta g k_\vartheta) = e^{il\theta}f(g) e^{in\vartheta},
\end{align}
for all $g \in G$, $k_\theta, k_{\vartheta}\in K$, where $l$ and $n$ are integers of the same parity. Please see Section \ref{Preliminaries} for any unexplained notation.
\begin{theorem} \label{lp-lq-mutli_int}
Let $1 < p \leq 2 \leq q < \infty$ and $l,n \in \Z^{\tau}$. Suppose $T_m$ is the multiplier operator defined in \eqref{mult_(l,n)} with the symbol $m$. Then there exists a constant $C=C({p,q}) > 0$, such that the following holds
    \begin{align*}
        \|T_m \|_{L^p(G)_{l,n} \rightarrow L^q(G)_{l,n}} \leq C \sup_{\alpha>0} \alpha \left( \int\limits_{\{  \lambda \in \R : |m(\lambda)|> \alpha\}}  \mu(\tau,\lambda) d\lambda \right)^{\frac{1}{p}-\frac{1}{q}} + C\left(\sum_{k \in i\Gamma_{l,n}}   |m(k)| |k|^{\frac{1}{p}-\frac{1}{q}}\right),
    \end{align*}
     where  $ \|T_m \|_{L^p(G)_{l,n} \rightarrow L^q(G)_{l,n}} $ denotes the operator norm of $ T_m$ from $L^p(G)_{l,n}$ to $L^q (G)_{l,n}$ and $\mu(\tau,\lambda)$ represents the density function \eqref{defn_mu_t} of the Plancherel measure associated with the principal series part.   
\end{theorem}
A classical result in harmonic analysis is the Hausdorff-Young inequality, which dictates that the Fourier transform maps functions from $L^p$ spaces to $L^{p'}$ spaces with respect to the Plancherel measure for $1\leq p\leq 2$. Paley raised a fundamental question: Given $p\in (1,2]$, can we find a measure $\nu$ such that the Fourier transform operator is bounded from $L^p$ to $L^p(\nu)$? To be more precise, the author established the following result within the context of Euclidean spaces, which also implies a weighted version of Plancherel’s formula for $L^p(\R^d)$ with $1<p\leq 2$.
\begin{theorem}[{{\cite[Theorem 1.10]{Hormander_Acta_60}}}]
Let $\phi>0 $ be a measurable function on $\R^d$ such that 
\begin{align*}
     \left|\left\{ \xi \in \R^d : |\phi(\xi)|>\alpha \right\} \right| \leq {C}/{\alpha},
\end{align*}
    for all $\alpha>0$. Then, for any $ 1<p\leq 2$, we have
    \begin{align*}
        \left( \int_{\R^d} |\mathcal{F}f(\xi)|^{p} \phi(\xi)^{2-p} \,d\xi \right)^{\frac{1}{p}}\leq C_p \|f\|_{L^p(\R^d)},
    \end{align*}
where the constant $C_p$ depends only on $p$ and $\phi$.
\end{theorem}
We have proven both the Paley and Hausdorff-Young inequalities for $(l,n)$ type functions on $G$. Moreover, by utilizing a version of the Stein-Weiss interpolation, we derived the Hausdorff–Young–Paley inequality in our setting. This inequality plays a crucial role in establishing $L^p$-$L^q$ estimates for the multiplier operator.

Moving forward, we explore the application of our main result. We establish $L^p$-$L^q$ estimates for the heat propagator of the $(l,n)$-th radial component of the Casimir element $\Omega$ (in $\S$ \ref{sec_application}). In contrast to the Euclidean setting or noncompact type Riemannian symmetric spaces (\cite[Theorem 3.2]{CGM93}), we observe growth in the solution of the heat equation near infinity within the context of the group $G$ (see Remark \ref{rem_dec_heat}). This growth arises from an additional discrete component of the Plancherel measure in $G$. Consequently, in the group scenario, the discrete part of the Plancherel measure contributes to the solution's growth as time progresses towards infinity in the $\Omega$-heat equation:
\begin{equation*}
\begin{aligned}
     u_t=\Omega u, \qquad
    u(0)=u_0 \in L^p(G)_{l,n}, 
\end{aligned}
\end{equation*}
for  $1<p\leq 2$, where $l,n \in \Z^{\tau}$.

Additionally, we establish the well-posedness of nonlinear abstract Cauchy problems in the space $L^\infty(0, T, L^2(G)_{l,n})$. Specifically, we initially consider the following heat equation:
\begin{equation} \label{heat_int}
u_t-|Bu(t)|^p=0, \quad u(0)=u_0,
\end{equation}
where $1<p<\infty$ and $B$ is a Fourier multiplier defined in \eqref{mult_(l,n)} on $L^2(G)_{l,n}$. We investigate the local well-posedness of the heat equation \eqref{heat_int} above. Next, we consider the initial value problem for the nonlinear wave equation:
\begin{align}\label{E-WNLE_int}
u_{tt}(t) - \Psi(t)|Bu(t)|^{p} = 0,
\end{align}
with the initial conditions $u(0)=u_0$ and $u_t(0)=u_1$, where $1< p<\infty$, $\Psi$ is a positive bounded function depending only on time, and $B$ is a linear operator initially defined in $L^2(G)_{l,n}$. We explore the global and local well-posedness of \eqref{E-WNLE_int} under certain conditions on the function $\Psi$.
  
  We conclude this section by providing an outline of this article.

In the upcoming section, we lay out the preliminaries on the group $G=\mathrm{SL}(2,\R)$ and discuss the spherical analysis of $(l,n)$ type functions on $G$. Following this, in Section \ref{sec_F_ineq}, we establish several key inequalities involving the Fourier transform for all $(l,n)$-type functions on $G$. Then, in Section \ref{main_result}, we define the multiplier operator within this context and prove our main result, Theorem \ref{lp-lq-mutli_int}. Finally, in Section \ref{sec_application}, we discuss various applications utilizing our main result.

\section{Preliminaries}\label{Preliminaries}
\subsection{Generalities}
The letters $\N$, $\Z$, $\R$, and $\C$ will respectively denote the set of all natural numbers, the ring of integers, and the fields of real and complex numbers. We denote the set of all non-negative integers and nonzero integers by $ \Z_+ $ and $\Z^*$, respectively. For $ z \in \C $, we use the notations $\Re z$ and $\Im z$ for real and imaginary parts of $z$, respectively. We shall follow the standard practice of using the letters $C$, $C_1$, $C_2$, etc., for positive constants, whose value may change from one line to another. Occasionally, the constants will be suffixed to show their dependencies on important parameters. We will use $X\lesssim Y$ or $Y\gtrsim X$ to denote the estimate $X\leq CY$ for some absolute constant $C>0$. We shall also use the notation $X\simeq Y$ for $X\lesssim Y$ and $Y\lesssim X$.  

\subsection{The group $ \mathrm{SL}(2,\R) $}\label{sec:The_group_SL(2,R)} From this section onwards, $G$ will always denote the group $\mathrm{SL}(2, \R)$, the group of $2\times 2$-matrices over $\R$ with determinant one. The Lie algebra of $G$, denoted by        $\mathfrak{g}$, can be realized as
\begin{align*}
    \mathfrak{g}= \mathrm{sl}(2,\R) =\left\{\begin{pmatrix}
				a  & b   \\
				c & -a
			\end{pmatrix} :  a,b,c \in \R \right\}.
\end{align*}
Important elements in $\mathfrak{g}$ are 
\begin{align*}
    X=\begin{pmatrix}
				0 & 1   \\
				-1 & 0
			\end{pmatrix},\,    
   H=\frac{1}{2}\begin{pmatrix}
				1 & 0   \\
				0 & -1
			\end{pmatrix},\,
   Y=\begin{pmatrix}
				0 & 1   \\
				0 & 0
			\end{pmatrix},  \, 
   \overline{Y}=\begin{pmatrix}
				0 & 0   \\
				-1 & 0
			\end{pmatrix}.   
\end{align*}
Corresponding elements in the group to the first three of these algebra elements are
\begin{align*}
    k_{\theta}&:=\begin{pmatrix}
				\cos \theta & \sin \theta   \\
				-\sin \theta & \cos \theta
			\end{pmatrix} =\exp{\left( \theta X \right)},\, \theta \in \R,\,    
   \qquad a_t := \begin{pmatrix}
				e^t  & 0   \\
				0 & e^{-t}  
			\end{pmatrix} =\exp{\left( 2t H  \right)},\, t \in \R,\\
 & \hspace{5cm}n_v := \begin{pmatrix}
				1  & v  \\
				0 & 1  
			\end{pmatrix} =  \exp{\left( v Y  \right)}, \, v \in \R.   
\end{align*}
Particular subgroups of $G$ are defined by
\begin{align*}
    K= \{ k_{\theta}: \theta \in  \R \}= \mathrm{SO}(2), \,    A=\left\{ a_t :  t \in \R \right\}, \text{ and } N=\left\{ n_v :  v \in \R \right\}.
\end{align*}
The Iwasawa decomposition for $G$ provides a diffeomorphism between $K \times A \times N$ and $G$. In other words, through the Iwasawa decomposition, any $x \in G$ can be uniquely expressed as $x = k_\theta a_t n_v$; using this, we write $$K(x) = k_\theta,\quad H(x) = t, \quad \text{and} \quad N(x)= v.$$ In fact, if $ x=\begin{pmatrix}
				a  & b  \\
				c & d  
			\end{pmatrix}\in \mathrm{SL}(2,\R)  $, then $\theta, t$ and $v$ are determined from
   \begin{equation}\label{formula_for_SL2R}
       e^{2t} =a^2+c^2,\, e^{i\theta}= \frac{a-ic}{\sqrt{a^2+c^2}}, \text{ and } v= \frac{ab+cd}{\sqrt{a^2+c^2}}.
   \end{equation}
        Next, let $A^+=\{ a_t \mid t>0 \}$; then the Cartan decomposition for $G$  gives $$G={K \overline{A^+}K} .$$ Using this we define $g^{+} $ as the $\R^+$ component in the Cartan decomposition, of the element  $g\in G$, that is we denote $a_{g^+}$ as the unique element such that $$g =k_1 a_{g^+} k_2, \quad \text{for some } k_1,k_2 \in K.$$
         For $ n \in \Z $, a complex-valued function $f$ on $G$ is said to be of  right  (resp. left) $K$-type $n$ if \begin{equation} \label{ K type n}
 f(xk) = f(x) e_n(k)\text{ (resp. } f(kx) =e_n(k) f(x) ) \text{ for all } k \in K \text{ and } x \in G,
  \end{equation} where $e_n(k_\theta) = e^{in \theta}$. For a class of functions $\mathcal{F} $ on $G$ (e.g., $L^1(G)$),  let $\mathcal{F}_n $ denote the corresponding subclass of functions of right $n$ type. We will denote $\mathcal{F}_{l,n} $ as the subclass of $\mathcal{F}_n$, which are also of left type $l$ and each element of $\mathcal{F}_{l,n} $ will be called a $ (l,n)  $ type function on $ G $.
    If $n=0$, we refer to such functions as $K$-biinvariant functions.

   By the Cartan decomposition we can extend a function $ f $ defined on $ \overline{A^+} $ to an $ (l,n) $ type function on $ G $ by 
   \begin{equation}\label{extension a function defined on closure of A+ to G}
				f( k_\theta a_t k_\vartheta ) = e_l(k_\theta) f(a_t) e_n(k_\vartheta), \quad \text{ for all } k_\theta , k_\vartheta \in K.
	\end{equation}  
 It is known that the following groups $G$, $ {N}$, and $K$ are unimodular, and we will denote the  (left or equivalently right) Haar measures of these groups as $dx$, $d {v}$ and  $dk$, where $\int_K dk=1$. Then we have the following integral formulae corresponding to the Cartan decomposition, which hold for any integrable function $f$:
			\begin{equation*}\label{definition of integration}
				\int_G f(x) dx = \int_K \int_{\R^+} \int_K f(k_1 a_t k_2) \Delta(t) dk_1 dt dk_2,
			\end{equation*}
			where $\Delta(t)= 2 \sinh 2t$.

\subsection{Spherical analysis of $(l,n)$ type functions}
		We are going to define Fourier transforms of suitable $(l,n)$ functions using representations of $G$ (see \cite[Chapter 9]{Barker}).

  Let $ \what M $ denote the equivalence classes of irreducible representations of $ M= \{  \pm I \} $, where $ I $ is the identity matrix. Then \[ \what M = \{ \tau_ +, \tau_-\}, \]
			where $ \tau_ +(\pm I)=1 $ and $ \tau_-(\pm I) =\pm 1 $. For each $\tau \in \what M$,  $\Z^{\tau} $ stands for the set of even integers for $\tau =\tau_+$, the set of odd integers for $\tau =\tau_-$ and $ -\tau $ will denote the opposite parity of $ \tau $.  
			
   \noindent We define for $l,n \in \Z^{\tau}$, \begin{align}\label{def_Gam_l,n}
				\Gamma_{l,n} = \lbrace
					k\in \Z^{-\tau} : 0 <k < \min\{l,n\}, \text { or } \max\{l,n\}< k<0 \rbrace.
			\end{align}
		For $\lambda \in \C $ and $ \tau \in \what M $,  let $(\pi_{\tau, \lambda }, H_\tau)$ be the principal series representation of $ G $ given by
		\begin{equation}\label{definition of principal series representation}
			{\left(\pi_{\tau, \lambda}(x) e_n\right)(k)=e^{(i\lambda - 1)H(xk)} e_n\left(K(x^{-1}k^{-1})\right)^{-1}}
		\end{equation}
		for all $ x \in G, k\in K $, where {$H_\tau$}  is the subspace of $L^2(K)$ generated by the orthonormal set  $\{ e_n \mid n \in \Z^\tau \}$. This representation is unitary if and only if $\lambda \in \R$.   For $ \lambda =0  $, the representation $\pi_{\tau_-, 0}$  has two irreducible subrepresentations, the so called  {mock discrete series}. We will denote them by $D_+$ and $D_-$. The representation spaces of  {$D_+$ and $D_-$ contain  $e_n \in L^2(K)$ respectively for positive odd $n$'s and negative odd $n$'s}.   For each  $k\in\mathbb Z^\ast$ (set of nonzero integers), there is a discrete series representation $\pi_{ik}$, which occurs as a subrepresentation of $\pi_{\tau, i|k|}$, $k\in \Z\setminus \Z^{\tau}$   (see \cite[p.19]{Barker}).  We define for    $ k \in \Z^*  $
		\begin{equation*}
			\Z(k) = \begin{cases}\{ m \in \Z^{\tau} : m \geq k+1 \} \text{ if }  k\geq 1 \\
				\{ m \in \Z^{\tau} : m \leq k-1 \} \text { if }  k\leq -1.
			\end{cases}
		\end{equation*}
		For $n\in \Z^{\tau}$,  the canonical matrix coefficient for the principal series
		\begin{equation}\label{definition of phi sigma lamda m, n}
			\phi_{\tau,\lambda}^{l,n}(x) := \left\langle \pi_{\tau, \lambda}(x) e_l, e_n \right\rangle  =\int_K e^{-(i\lambda+1)H(xk)} e_l(k^{-1}) \overline{e_n(K(xk)^{-1})} dk,  
		\end{equation}
		are functions of $ (n,l) $ type.   We note  that $\phi_{\tau_+, \lambda}^{0,0}$ corresponds to the elementary spherical function commonly denoted as $\phi_\lambda$.
  
  For $k \in \Z^*$ and $n\in \Z(k)$ the canonical matrix coefficient of the discrete series representation is
            \begin{equation*} {\psi^{l,n}_{ik}(x):=\langle \pi_{ik}(x) e_l^k, e_n^k\rangle_k}, 
		\end{equation*}
		where ${e_n^k}$ are the renormalized basis and $\langle \,,\rangle_k$ is the renormalized inner product for $\pi_{ik}$, see \cite[p. 20]{Barker} for more details. These functions $\phi_{\tau,\lambda}^{l,n}$ and $\psi^{l,n}_{ik}$  are  also known as spherical functions of  $(n,l)$ type. We will denote $ H_k $ as the Hilbert space generated by $ \{ e_n^k : n \in \Z(k)\} $.   It is known  \cite[Prop. 7.3]{Barker} that 
         \begin{equation}\label{relation between psi n,n and phi {n,n}}
			\psi_{ik}^{l,n}(x)=  \eta^{l,n}(k) \phi_{\tau, i|k|}^{l,n}(x), \quad \text{for all } x \in G,
		\end{equation} where $\tau \in \what{M}$ is determined by $ k \in \Z^{-\tau} $ and $\eta^{l,n}(k)$ is a constant does not depend on $x$. Furthermore, we have the following estimate \cite[(8.1)]{Barker} of $\psi_{ik}^{l,n}$. This estimation constitutes a minor yet crucial extension of a result initially presented by Trombi and Varadarajan in the $\mathrm{SL}(2,\R)$ case.
\begin{theorem}
    Let $k\in \N$, There exist constants $C, r_1, r_2, r_3, r_4\geq 0$ such that
   \begin{align}\label{est_psi_l,n}
       |\psi_{ik}^{l,n}(x)| \leq C (1+|l|)^{r_1}(1+|n|)^{r_2}(1+|k|)^{r_3}(1+ |H(x)|)^{r_4} \phi_{\tau_+, 0}^{0,0}(x)^{1+l}.
   \end{align}
\end{theorem}
 Additionally, we recall that from \cite[(3.5)]{ADY_96} (see also\cite[(3.2)]{Barker}) for all $t\in \R$,
\begin{align}\label{est_phi_0,0}
    \phi_{\tau_+, 0}^{0,0}(a_t)\asymp (1+|t|) e^{-|t|}.
\end{align}

 For a given $\tau \in \what{M}$ and $\lambda \in \C$, we define the infinitesimal representation of $ \mathfrak{g}$ and $K$ associated to  $(\pi_{\tau, \lambda }, H_\tau)$ as follows
\begin{align*}
    \pi_{\tau, \lambda }(X)v: = \frac{d}{dt } \pi_{\tau, \lambda }( \exp (tX)) v\big|_{t=0}
\end{align*}
 for $X\in \mathfrak{g}$, $v \in H_{\tau}$. Let $\Omega$ denote the Casimir element on $G$ defined by 
 (\cite[(2.6)]{Barker})
 \begin{align} \label{Casimirele}
     \Omega = H^2+H- \overline{Y} Y.
 \end{align}
 It is known that the principal series representations act as scalars on the Casimir element $\Omega$; in fact 
  \begin{equation*} 
    \pi_{\tau, \lambda }(\Omega)  = -\frac{\lambda^2+1}{4} .
    \end{equation*}
 Moreover,  the canonical matrix coefficient of principal series $\phi_{\tau, \lambda}^{l,n}$ satisfies the differential equation 
    \begin{equation}\label{ef_phi_n,n}
    \Omega f = -\frac{\lambda^2+1}{4} f,
    \end{equation}
     where $\Omega$ acts as a biinvariant differential operator on $G$, please see \cite[(4.7)]{Barker}.

  We remark that we use a different parameterization of the representations and spherical functions from Barker \cite{Barker}. According to our definition in \eqref{definition of principal series representation}, the unitary dual of $\mathrm{SL}(2,\R)$ is $\R$, while according to Barker's convention, the unitary dual of $\mathrm{SL}(2,\R)$ is $i\R$. As a result, our $\pi_{\lambda}$ corresponds to his $\pi_{-i\lambda}$, and similarly, $\phi_{\tau,\lambda}^{l,n}$ and $\psi_{ik}^{l,n}$ are reparametrized accordingly.   This choice of parametrization offers a clearer analogy with the general semisimple Lie groups case, making our analysis more transparent.
     
   Let $l,n\in \Z^\tau.$ Then for a smooth compactly supported $(l,n)$ type function $f$  on $G$, the principal series Fourier transform of $f$ is defined by  \begin{equation*}
			\what{f_H}(\lambda) = \int_G f(x) {\phi_{\tau,\lambda}^{l,n}}(x^{-1}) dx,
		\end{equation*}
   for all $\lambda \in \R$, and  the discrete series Fourier transform is defined by 
        \begin{equation*}
			\what{f_B}(k)  = \int_G f(x) {\psi_{k}^{l,n}}(x^{-1}) dx,
		\end{equation*}
for all $k \in i\Gamma_{l,n}$. Then the inversion formula is given by \cite[Theorem 10.5]{Barker} (see also \cite{MR0047055}):
\begin{equation}\label{inversion_l,n}
    f(x)= \frac{1}{4\pi^2} \int_{\R} \what {f}_H(\lambda) \phi^{n,l}_{\tau,\lambda}(x)  \mu(\tau,\lambda) d\lambda + \frac{1}{2\pi} \sum_{k \in i\Gamma_{l,n}} \what{f}_B (k) \psi^{n,l}_{k}(x) |k|, 
\end{equation}
where $\tau \in \what M$ is determined by $n \in \Z^\tau$, and $\mu(\tau,\lambda)$ is given by \cite[(10.1)]{Barker}
    \begin{align}\label{defn_mu_t}
	\mu(\tau,\lambda)  = \begin{cases}
	\frac{\lambda \pi} {2} \tanh (\lambda \pi /2) , &\text{ if } \tau =\tau_+\\
	\frac{\lambda \pi}{2} \coth (\lambda \pi /2) , &\text{ if } \tau =\tau_-,
	 	\end{cases}
			 \end{align}
  for all $\lambda \in \R$.  We also have the following Plancherel formula from \cite[Theorem 2]{MR1011896} (see also \cite[page 47]{Barker}), which holds for all $f \in L^1 (G)_{l,n}\cap L^2(G)_{l,n}$.
\begin{equation}\label{plancherel,l,n}
    \int_G |f(x)|^2 \, dx = \frac{1}{4\pi^2} \int_{\R} |\what {f}_H(\lambda)|^2  \mu(\tau,\lambda)\, d\lambda + \frac{1}{2\pi} \sum_{k \in i\Gamma_{l,n}} |\what{f}_B (k)|^2   |k|.
\end{equation}

\section{Fourier inequalities on all matrix coefficients on $\mathrm{SL}(2,\R)$}\label{sec_F_ineq}
In this section, we present the analog version of some classical inequalities mentioned in the introduction. We will start by introducing the Hausdorff-Young inequality for $(l,n)$-type functions on $G$.  To do this, let us consider the measure space $\left(\R \cup i\Gamma_{l,n},d{\tilde{\nu}}_{l,n} \right)$, where $d{\tilde{\nu}}_{l,n}$ is defined by 
    \begin{align*}
        \int_{\R\cup i \Gamma_{l,n}} F(\zeta) d{\tilde{\nu}}_{l,n}(\zeta) =  \frac{1}{4\pi^2}  \int_{\R} F(\lambda)   \mu(\tau,\lambda) d\lambda +  \frac{1}{2\pi}\sum_{k\in i\Gamma_{l,n}}  F(k)  |k|,
    \end{align*}
    for any suitable function $F$ on $\R \cup i\Gamma_{l,n}$. For  $l,n \in \Z^{\tau}$, we also define the operator $\mathcal{F}_{l,n}$ on $ C_c^{\infty}(G)_{l,n}$ by 
    \begin{align}\label{def_F}
        \mathcal{F}_{l,n}f(\zeta)= \begin{cases}
            \what{f}_H(\zeta) \quad \text{if }\zeta \in \R,\\
            \\
            \what{f}_B(\zeta) \quad \text{if }\zeta \in i\Gamma_{l,n},
        \end{cases} 
    \end{align}
    for $f \in C_c^{\infty}(G)_{l,n}$.  Then we have the following inequality. For $1 \leq p \leq \infty$, let $p' = p/(p-1)$ denotes the conjugate exponent of $p$. When $p=\infty$, by $(\int |F|^{p})^{\frac{1}{p}}$, will represent the $L^{\infty}$ norm of $F$.
    \begin{theorem}\label{thm_hy}
Let $1\leq p\leq 2$ and $l,n \in \Z^{\tau}$. Then there exists a constant $C=C(p.l,n)>0$ such that
\begin{align*}
\left( \int_{\R\cup i \Gamma_{l,n}}  \left| \mathcal{F}_{l,n}f(\zeta) \right|^{p'} d{\tilde{\nu}}_{l,n}(\zeta) \right)^{\frac{1}{p'}} \leq C \|f\|_{L^p(G)}
\end{align*}
for all $f \in C_c^{\infty}(G)_{l,n}$. Moreover, for all $f \in C_c^{\infty}(G)_{l,n}$, there exists a constant $C=C(p.l,n)>0$ such that
\begin{align}\label{inv_Hausd}
    \left(\int_G |f(x)|^{p'} dx\right)^{\frac{1}{p'}} \leq C \left(\int_{\R\cup i \Gamma_{l,n}}  |\mathcal{F}_{l,n}f(\zeta)|^p d{\tilde{\nu}}_{l,n}(\zeta) \right)^{\frac{1}{p}}.
\end{align}
\end{theorem}
\begin{proof}
   Let $f \in C_c^{\infty}(G)_{l,n}$. By the Plancherel formula \eqref{plancherel,l,n},  we have 
    \begin{equation}\label{Pla_form}
        \begin{aligned}
        \int_{\R\cup i \Gamma_{l,n}} |\mathcal{F}_{l,n}f(\zeta)|^2d{\tilde{\nu}}_{l,n}(\zeta)&= \frac{1}{4\pi^2}  \int_{\R} |\what{f}_H(\lambda)|^2  \mu(\tau,\lambda) d\lambda +  \frac{1}{2\pi}\sum_{k\in i\Gamma_{l,n}}  |\what{f}_B(k)|^2  |k|\\
        &=\int_G |f(x)|^2 dx,
    \end{aligned}
    \end{equation}
    whence we obtain that the operator $\mathcal{F}_{l,n}$ is bounded from $L^2(G)_{l,n}$ into $L^2(\R \cup i\Gamma_{l,n}, d\tilde{\nu}_{l,n})$. 
     We can write from \eqref{definition of phi sigma lamda m, n} that for $\lambda \in \R$
     \begin{align*}
         |\phi_{\tau,\lambda}^{l,n}(x)|&= \left|\int_K e^{-(i\lambda+1)H(xk)} e_l(k^{-1}) \overline{e_n(K(xk)^{-1})} dk \right|\\
         & \leq \int_K e^{-H(xk)}dk\\
         &= \phi^{0,0}_{\tau^+, 0}(x),
     \end{align*} 
    which, by the Helgason-Johnson theorem, is bounded. Moreover, we have
        \begin{align*}
    |\phi^{0,0}_{\tau^+, \lambda}(x)| \leq 1
        \end{align*}
        for all $\lambda \in \R$ and $x \in G$. Thus, we obtain 
    \begin{align}\label{est_inf_c}
        \sup_{\lambda \in \R}|\what{f}_{H}(\lambda)| \leq  \|f\|_{L^1(G)_{l,n}}
    \end{align}
 for all $\lambda \in \R$.   Similarly, using the estimate \eqref{est_psi_l,n} of 
 $\psi_{ik}^{l,n}$ and \eqref{est_phi_0,0}, we get
    \begin{align}\label{est_inf_d}
        \max_{k\in i\Gamma_{l,n}} |\what{f}_B(k)| \leq C_{l,n}\|f\|_{L^1(G)_{l,n}}.
    \end{align}
    Therefore, inequalities \eqref{est_inf_c} and \eqref{est_inf_d} imply that $\mathcal{F}_{l,n}$ is a bounded map from $L^1(G)_{l,n}$ to $L^{\infty} (\R \cup i\Gamma_{l,n})$. Finally, by leveraging the properties that the operator $\mathcal{F}$ exhibits of type $(1,\infty)$ and $(2,2)$, we achieve the Hausdroff-Young inequality through the application of the Riesz-Thorin interpolation theorem.

    Next to prove \eqref{inv_Hausd} again using the estimates of the spherical functions $\phi^{l,n}_{\tau,\lambda}$ and $\psi_{ik}^{l,n}$, for $f\in C_c^{\infty}(G)_{l,n}$ we have 
    \begin{align}\label{eqn_1,inf}
        \|f\|_{L^{\infty}(G)}\leq  \frac{1}{4\pi^2} \int_{\R}  |\what{f}_H(\lambda)|  \mu(\tau,\lambda) d\lambda 
          +\frac{1}{2\pi}\sum_{k\in i\Gamma_{l,n}} |\what{f}_B(k)|   |k| = \int\limits_{\R \cup i\Gamma_{l,n}} |\mathcal{F}_{l,n}(f)(\zeta)| d {\tilde{\nu}_{l,n}(\zeta)},
    \end{align}
and then again using the Riesz-Thorin interpolation to the inequalities \eqref{Pla_form} and \eqref{eqn_1,inf} we obtain the desired inequality.   
\end{proof}

We now consider a Paley-type inequality for $(l,n)$-type functions on $G$.
\begin{theorem}\label{thm_paley}
 Let $1< p\leq 2$. Assume that $\psi$ is a positive function on $\R\cup i\Gamma_{l,n}$ satisfying the following condition 
	\begin{align*}
	 	 \|\psi\|_{\tau,\infty} : = \sup_{\alpha>0} \alpha  \int\limits_{\substack{\{\lambda \in \R : \psi(\lambda)>\alpha\} }} \mu(\tau,\lambda) \,d\lambda <\infty.
	 \end{align*}
  Then we have 
  \begin{align*}
          \left(\frac{1}{4\pi^2}\int_{\R}   {|\what {f}_{H}(\lambda)|}  ^{p}  \psi(\lambda)^{2-p}  \mu(\tau,\lambda)\,  d\lambda+  \frac{1}{2\pi}\sum_{k\in i\Gamma_{l,n}}   |\what {f}_{B}(k) |^p   \psi(k)^{2-p} |k|\right)^{\frac{1}{p}} & \\
          &\hspace{-4.5cm}\lesssim \left(   \|\psi\|_{\tau,\infty}+   \sum_{k \in i\Gamma_{l,n}}   \psi(k) |k|    \right)^{\frac{2-p}{p}} \|f\|_{L^p(G)},
     \end{align*}
     for all $f\in L^p(G)_{l,n}$.
\end{theorem}
\begin{proof}
    Let us define a measure on $\R \cup i\Gamma_{l,n}$ by the following:
    \begin{align}\label{meas_dual}
        \int_{\R\cup i \Gamma_{l,n}} F(\zeta) d\vartheta_{l,n}(\zeta) =  \frac{1}{4\pi^2}  \int_{\R} F(\lambda) \psi(\lambda)^2 \mu(\tau,\lambda)\,d\lambda +  \frac{1}{2\pi}\sum_{k\in i\Gamma_{l,n}} F(k) \psi(k)^2|k|,
    \end{align}
    for any suitable function $F$ on $\R \cup i\Gamma_{l,n}$. Then it is easy to check that  $  \vartheta_{l,n}$ is a well-defined measure on  $\R \cup i\Gamma_{l,n}$.
Now, let us define an operator 
\begin{align}\label{defn_of_T}
    Tf(\zeta) = \begin{cases}
       \frac{ \what{f}_H(\zeta)}{\psi(\zeta)}, \quad \text{if } \zeta \in \R,\\
\\
        \frac{\what{f}_B(\zeta)}{\psi(\zeta)}, \quad \text{if } \zeta \in i\Gamma_{l,n}.
    \end{cases}
\end{align}
Then by the Plancherel theorem we have for $f\in C_c^{\infty}(G)_{l,n}$,
\begin{align*}
     \int_{\R\cup i\Gamma_{l,n}}  |Tf(\zeta)|^2  \psi(\zeta)^2 d\vartheta_{l,n}(\zeta) &= \frac{1}{4\pi^2}  \int_{\R} \frac{|\what{f}_H(\lambda)|^2}{\psi(\lambda)^2} \psi(\lambda)^2 \mu(\tau,\lambda)\,d\lambda +  \frac{1}{2\pi}\sum_{k\in i\Gamma_{l,n}} \frac{|\what{f}_B(k)|^2}{\psi(k)^2} \psi(k)^2 |k|\\
    &=  \frac{1}{4\pi^2}  \int_{\R} |\what f_H(\lambda)|^2 \mu(\tau,\lambda) \,d\lambda +  \frac{1}{2\pi}\sum_{k\in i\Gamma_{l,n}} |\what {f}_B(k)|^2|k|\\
    & = \|f\|_{L^2(G)}^2.
\end{align*}
Thus, we have established that $T$ is a bounded operator from $L^2(G)_{l,n}$ to $L^2(\mathbb{R} \cup i\Gamma_{l,n}, $ $\psi(\zeta)^2 d\vartheta_{l,n}(\zeta))$. Our next objective is to demonstrate that $T$ is also weak-type $(1,1)$. To achieve this, we first decompose the operator $T$ into two parts, namely $T^{\text{con}}$ and $T^{\text{dis}}$, defined by

\begin{align*}
    T^{\text{con}} f(\lambda) =  \frac{ \what{f}_H(\zeta)}{\psi(\zeta)} \quad \text{for }\lambda \in \R, \quad \text{and }   \quad T^{\text{dis}} f(k)= \frac{\what{f}_B(k)}{\psi(k)} \quad \text{for }k \in i\Gamma_{l,n}.
\end{align*}
Since $|\what f_B(ik)|\leq \|f\|_{L^1(G)}$ for all $f\in L^1(G)_{l,n}$, we have 
\begin{align*}
    \sum_{k \in i\Gamma_{l,n}} |T^{\text{dis}} f(k)| \psi(k)^2\, |k| =  \sum_{k \in i\Gamma_{l,n}} |\what {f}_B(k)| \psi(k) |k| \leq \left( \sum_{k \in i\Gamma_{l,n}}   \psi(k) |k|\right) \|f\|_{L^1(G)}  .
\end{align*}
Next, we aim to establish the inequality: 
\begin{align}\label{eqn_w1,1}
    \vartheta_{l,n}\left\{ \lambda \in \R :  \frac{ |\what{f}_H(\lambda)|}{\psi(\lambda)} >\alpha  \right\}\lesssim\|\psi\|_{\tau,\infty} {\|f\|_{L^1(G)}}/{\alpha}
\end{align}
for $\alpha>0$ and $f\in L^1(G)_{l,n}$. Recall from \eqref{definition of phi sigma lamda m, n} that $|\phi^{l,n}_{\tau,\lambda}(x)| \leq \phi^{0,0}_{\tau^+,\lambda}(x)$. Using this, one obtains:
\begin{align*}
    |\what{f}_H(\lambda)|\leq \|f\|_{L^1(G)} \quad \text{for all }\lambda \in \R.
\end{align*}
Therefore, \eqref{eqn_w1,1} follows from the subsequent inequality:
\begin{align}\label{eqn_w1,1_f}
    \vartheta_{l,n}\left\{ \lambda \in \R :  \frac{ \|f\|_{L^1(G)}}{\psi(\lambda)} >\alpha  \right\}\lesssim \|\psi\|_{\tau,\infty} {\|f\|_{L^1(G)}}/{\alpha}
\end{align}
for all $\alpha >0$. We observe that it suffices to prove the inequality above for $\|f\|_{L^1(G)} =1$. By using the definition of the measure $ \vartheta_{l,n}$, we get
\begin{align}\label{eqn_cov_Msi}
     \vartheta_{l,n}\left\{ \lambda \in \R :   \psi(\lambda)<\frac{1}{\alpha}  \right\}=   \int\limits_{\{\lambda\in \R:  \psi(\lambda)<\frac{1}{\alpha}\}} \psi(\lambda)^2 \mu(\tau,\lambda) \,d\lambda,
\end{align}
which we can rewrite as follows:
\begin{align*}
    \int\limits_{\{\lambda\in \R:  \psi(\lambda)<\frac{1}{\alpha}\}} \psi(\lambda)^2 \mu(\tau,\lambda)  \,d\lambda=  \int\limits_{\{\lambda\in \R:  \psi(\lambda)<\frac{1}{\alpha}\}} \left( \int_0^{\psi(\lambda)^2} 
 d\xi \right) \mu(\tau,\lambda) \,d\lambda.
\end{align*}
Next, by interchanging the order of integration, we obtain:
\begin{align*}
     \int\limits_{\{\lambda\in \R:  \psi(\lambda)<\frac{1}{\alpha}\}} \left( \int_0^{\psi(\lambda)^2} 
 d\xi \right)  \mu(\tau,\lambda)  \,d\lambda = \int_0^{\frac{1}{\alpha^2}} \left( \int\limits_{\{\lambda\in \R:  \xi^{\frac{1}{2}}<\psi(\lambda)<\frac{1}{\alpha}\}} \mu(\tau,\lambda)  \,d\lambda \right) d\xi,  
\end{align*}
which by the change of variable $\xi\rightarrow \xi^2$ transforms into 
\begin{align}\label{eqn_cov_Msi}
    2 \int_0^{\frac{1}{\alpha}} \xi \left( \int\limits_{\{\lambda\in \R:  \xi<\psi(\lambda)<\frac{1}{\alpha}\}} \mu(\tau,\lambda)  \,d\lambda \right) \, d\xi \lesssim  \int_0^{\frac{1}{\alpha}} \xi \left( \int\limits_{\{\lambda\in \R:  \psi(\lambda)> \xi\}}\mu(\tau,\lambda)  \,d\lambda \right) \, d\xi.
\end{align}
     Since, 
     \begin{align*}
         \xi \left( \int\limits_{\{\lambda\in \R:  \psi(\lambda)> \xi\}} \mu(\tau,\lambda)  \,d\lambda \right)\leq \sup_{\xi>0 }\, \xi \left( \int\limits_{\{\lambda\in \R:  \psi(\lambda)> \xi\}} \mu(\tau,\lambda)  \,d\lambda \right)=\|\psi\|_{\tau,\infty},
     \end{align*}
     so by using the hypothesis and \eqref{eqn_cov_Msi}, we have finally
     \begin{align*}
          \vartheta_{l,n}\left\{ \lambda \in \R :   \psi(\lambda)<\frac{1}{\alpha}  \right\}\leq  C \frac{\|\psi\|_{\tau,\infty}}{\alpha} 
     \end{align*}
     for all $\alpha>0$, where the constant $C$ does not depend on $l,n$. This concludes the proof of \eqref{eqn_w1,1_f} and consequently we have \eqref{eqn_w1,1}. Thus we have shown that the operator $T$ is bounded from $L^1(G)_{l,n}$ to $L^{1,\infty}(\mathbb{R} \cup i\Gamma_{l,n}, \psi(\zeta)^2 d\vartheta_{l,n}(\zeta))$. Next, utilizing the Marcinkiewicz interpolation theorem with $p_1=1$, $p_2=2$, and $\frac{1}{p} = 1-\theta+\frac{\theta}{2}$ ($0<\theta<1$), we proceed to obtain
     \begin{align*}
          \left(\frac{1}{4\pi^2}\int_{\R} \left( \frac{|\what {f}_{H}(\lambda)|}{\psi(\lambda)} \right)^{p}  \psi(\lambda)^2 \mu(\tau,\lambda)\,  d\lambda+  \frac{1}{2\pi}\sum_{k\in i\Gamma_{l,n}} \left( \frac{|\what {f}_{B}(k)|}{\psi(k)} \right)^{p} \psi(k)^2 |k|\right)^{\frac{1}{p}} & \\
          &\hspace{-4.5cm}\lesssim_p \left(\|\psi\|_{\tau,\infty}+ \sum_{k \in i\Gamma_{l,n}}   \psi(k) |k|  \right)^{\frac{2-p}{p}} \|f\|_{L^p(G)},
     \end{align*}
     for all $f\in L^p(G)_{l,n}$.
\end{proof}
Next, we document the subsequent interpolation theorem from \cite{MR0482275} for future reference.

\begin{theorem} \label{interpolation} Let $d\mu_0(\zeta)= \omega_0(\zeta) d\mu'(\zeta),$ $d\mu_1(\zeta)= \omega_1(\zeta) d\mu'(\zeta),$ and write $L^p(\omega)=L^p(\omega\, d\mu')$ for any weight $\omega.$ Suppose that $0<p_0, p_1< \infty.$ Then 
$$(L^{p_0}(\omega_0), L^{p_1}(\omega_1))_{\theta, p}=L^p(\omega),$$ where $0<\theta<1, \, \frac{1}{p}= \frac{1-\theta}{p_0}+\frac{\theta}{p_1}$ and $\omega= \omega_0^{\frac{p(1-\theta)}{p_0}} \omega_1^{\frac{p\theta}{p_1}}.$
\end{theorem} 

The following corollary is immediate.

\begin{corollary}\label{interpolationoperator} Let $d\mu_0(\zeta)= \omega_0(\zeta) d\mu'(\zeta),$ $d\mu_1(\zeta)= \omega_1(\zeta) d\mu'(\zeta).$ Suppose that $0<p_0, p_1< \infty.$  If a continuous linear operator $A$ admits bounded extensions, $A: L^p(Y,\mu)\rightarrow L^{p_0}(\omega_0) $ and $A: L^p(Y,\mu)\rightarrow L^{p_1}(\omega_1) ,$   then, we there exists a bounded extension $A: L^p(Y,\mu)\rightarrow L^{b}(\omega) $ of $A$, where  $0<\theta<1, \, \frac{1}{b}= \frac{1-\theta}{p_0}+\frac{\theta}{p_1}$ and 
 $\omega= \omega_0^{\frac{b(1-\theta)}{p_0}} \omega_1^{\frac{b\theta}{p_1}}.$
\end{corollary}
Now, we employ the preceding theorem to establish a Hausdorff–Young–Paley inequality through interpolation between the Hausdorff–Young inequality and the Paley inequality.
\begin{theorem}\label{thm_H_Y_P}
    Let $1<p\leq 2$ and $1<p\leq b \leq p'<\infty $. Assume that $\psi: \R \cup i\Gamma_{l,n} \rightarrow (0,\infty)$ is a positive function such that
    \begin{align}{\label{weak_inf}}
         \|\psi\|_{\tau,\infty} = \sup_{\alpha>0} \alpha  \int\limits_{\substack{\{\lambda \in \R : \psi(\lambda)>\alpha\} }} \mu(\tau,\lambda) \,d\lambda 
    \end{align}
    is finite. Then for all $f \in L^p(G)_{l,n}$, we have 
     \begin{align*}
          \left(\frac{1}{4\pi^2}\int_{\R} \left( {|\what {f}_{H}(\lambda)|}{\psi(\lambda)}^{\frac{1}{b}-\frac{1}{p'}} \right)^{b}     \mu(\tau,\lambda)\,  d\lambda+  \frac{1}{2\pi}\sum_{k\in i\Gamma_{l,n}} \left(  {|\what {f}_{B}(k)|}{\psi(k)}^{\frac{1}{b}-\frac{1}{p'}} \right)^{b} |k|\right)^{\frac{1}{b}} & \\
          &\hspace{-5cm}\lesssim_p \left( \|\psi\|_{\tau,\infty}+  \sum_{k \in i\Gamma_{l,n}}   \psi(k) |k|  \right)^{\frac{1}{b}-\frac{1}{p'}} \|f\|_{L^p(G)}.
     \end{align*}
\end{theorem}
\begin{proof}
 Assuming that the function $\psi$ satisfies the condition \eqref{weak_inf}, and recalling the operator $\mathcal{F}_{l,n}$ from \eqref{def_F}, we can utilize Theorem \ref{thm_paley} and Theorem \ref{thm_hy}. Consequently, we conclude that $\mathcal{F}_{l,n}$ is a bounded operator from $L^p(G)_{l,n}$ to $L^{p}(\R \cup i\Gamma_{l,n}, \psi^{2-p}d{\tilde{\nu}}_{l,n})$ and from $L^p(G)_{l,n}$ to $L^{p'}\left(\R \cup i\Gamma_{l,n}, d{\tilde{\nu}}_{l,n}) \right)$. Therefore, applying Corollary \ref{interpolationoperator} with $p_0=p$, $p_1=p'$,
      \begin{align*}
        d\mu_0(\zeta) = \psi(\zeta)^{2-p} d{\tilde{\nu}}_{l,n}(\zeta), \quad \text{and} \quad d\mu_1(\zeta) =  d{\tilde{\nu}}_{l,n}(\zeta),
    \end{align*} 
    we obtain that $\mathcal{F}_{l,n}$ is a bounded operator from $L^p(G)_{l,n}$ to $L^{b}\left(\R \cup i\Gamma_{l,n}, \omega(\zeta) d{\tilde{\nu}}_{l,n}(\zeta)\right)$, where
\begin{align*}
\frac{1}{b}= \frac{1-\theta}{p}+\frac{\theta}{p'} \quad \text{and} \quad \omega(\zeta) = \psi(\zeta)^{(2-p)\frac{b(1-\theta)}{p}},
\end{align*}
for $0<\theta<1$. Solving the expression above, we get $\theta= \frac{b-p}{(2-p)b}$, which implies
\begin{align*}
\omega(\zeta) = \psi(\zeta)^{1-\frac{b}{p'}}.
\end{align*}
Finally, together with the expression above of $\omega(\zeta)$ and the fact  $\mathcal{F}_{l,n}$ is a bounded operator from $L^p(G)_{l,n}$ to $L^{b}\left(\R \cup i\Gamma_{l,n}, \omega(\zeta) d{\tilde{\nu}}_{l,n}(\zeta)\right)$, the theorem above follows.
\end{proof}

\section{Boundedness of Fourier and spectral multipliers on $ \mathrm{SL}(2,\R) $}\label{main_result}

In this section, we will present our main results regarding the multiplier operators for $(l,n)$ type functions on $G$. Before delving into that, let us define the multiplier operator in our setting and briefly recall the $L^p$-bounded result for $(n,n)$-type multipliers by Ricci and Wróbel \cite{Ricci}.

Let $l,n \in \Z^{\tau}$, for some $\tau \in \{\tau_+, \tau_- \}$, and $m: \mathbb{R}\cup i\Gamma_{l,n} \rightarrow \C$ be a bounded measurable function. The associated Fourier multiplier operator $T_m$ for $(l,n)$ type functions on $G$ is defined through the inversion formula \eqref{inversion_l,n} as follows:
\begin{align}\label{mult_(l,n)}
T_m f(x) &= \frac{1}{4\pi^2} \int_{\R} m(\lambda) \what {f_H}(\lambda) \phi_{\tau, \lambda}^{n,l} (x) \mu(\tau,\lambda) d\lambda
+\frac{1}{2\pi}\sum_{k\in i \Gamma_{l,n}} m(k) \what{ f_B}(k) \psi_{k}^{n,l}(x) |k|.
\end{align}
In their work \cite{Ricci}, the authors tackled the $L^p(G)_{n,n}$-boundedness problem for $p\in (1,\infty) \setminus {2}$. They observed that if $T_m$ is bounded on $L^p(G)_{n,n}$ for some $p\in (1,\infty) \setminus {2}$, then the multiplier $m$, initially defined on $\R \times i\Gamma_{n,n}$, must extend to a bounded even holomorphic function in the interior of $S_{p} := \{ \lambda \in \C : |\Im \lambda| \leq |2/p-1|\}$. We have proven the $L^p(G)_{l,n}$-$L^q(G)_{l,n}$ estimates for $T_m$ with $1<p\leq 2\leq q<\infty$ and any $l,n \in \Z^{\tau}$ without assuming any regularity on the multiplier. For the sake of convenience of the reader, we recall one of our main results from Theorem \ref{lp-lq-mutli_int}.
\begin{theorem} \label{lp-lq-mutli}
    Let $1<p\leq 2 \leq q<\infty$ and $l,n \in \Z^{\tau}$. Suppose that $T_m$ is the multiplier operator defined in \eqref{mult_(l,n)} with the multiplier $m$. Then there exists a constant $C=C({p,q,l,n})> 0$, such that the following holds
    \begin{align*}
        \|T_m \|_{L^p(G)_{l,n} \rightarrow L^q(G)_{l,n}} \leq C \sup_{\alpha>0} \alpha \left( \int\limits_{\{  \lambda \in \R : |m(\lambda)|> \alpha\}}  \mu(\tau,\lambda) d\lambda \right)^{\frac{1}{p}-\frac{1}{q}} + C \left(\sum_{k \in i\Gamma_{l,n}}   |m(k)| |k| ^{\frac{1}{p}-\frac{1}{q}}\right) .
    \end{align*}
\end{theorem}
\begin{proof}
We first observe that, by duality, it suffices to prove the result for $p\leq q'$. Specifically, if $q'\leq p$, then the $L^p$-$L^q$ boundedness of $T_m$ implies the $L^{q'}$-$L^{p'}$ boundedness of $T_m^*= T_{\overline{m}}$. Moreover, $\| T_m\|_{L^p\rightarrow L^q} = \| T_m^* \|_{L^{p'}\rightarrow L^{q'}}$. Let us assume that $p\leq q'$. As $q\geq 2$, we can apply \eqref{inv_Hausd} to obtain
    \begin{align*}
        \|T_m f \|_{L^q(G)_{l,n}} \lesssim \left( \int\limits_{\R \cup i\Gamma_{l,n}} |\mathcal{F}_{l,n} (T_m f)(\zeta)|^{q'} d{\tilde{\nu}}_{l,n}(\zeta)  \right)^{\frac{1}{q'}}.  
    \end{align*}
Now by applying Theorem \ref{thm_H_Y_P} with $\psi =|m|^{r}$, $b= q'$ and letting $1/r =1/p-1/q$, it follows that
    \begin{equation}\label{ineq_Fl,nT}
        \begin{aligned}
         \left( \int\limits_{\R \cup i\Gamma_{l,n}} |\mathcal{F}_{l,n} (T_m f)(\zeta)|^{q'} d{\tilde{\nu}}_{l,n}(\zeta)  \right)^{\frac{1}{q'}}  &\lesssim \left( \left(\sup\limits_{\alpha>0}\,\alpha \int\limits_{\{ \lambda \in \R : |m(\lambda)|^r>\alpha \}   } \mu(\tau,\lambda) d\lambda\right)^{\frac{1}{r}} \right. \\&\hspace{2.5cm} \left. + \left( \sum_{k \in i\Gamma_{l,n}}   |m(k)|^{r} |k|\right)^{\frac{1}{r}}  
         \right) \|f\|_{L^p(G)_{l,n}},
    \end{aligned}
    \end{equation}
    for all $f \in L^p(G)_{l,n}$. Moreover, we can write 
    \begin{align*}
        \left(\sup\limits_{\alpha>0}\,\alpha \int\limits_{\{ \lambda \in \R : |m(\lambda)|^r>\alpha \}   } \mu(\tau,\lambda) d\lambda\right)^{\frac{1}{r}} &= \left(\sup\limits_{\alpha>0}\,\alpha \int\limits_{\{ \lambda \in \R : |m(\lambda)|>\alpha^{\frac{1}{r}} \}   } \mu(\tau,\lambda) d\lambda\right)^{\frac{1}{r}} \\
        & = \left(\sup\limits_{\alpha>0}\,\alpha^r \int\limits_{\{ \lambda \in \R : |m(\lambda)|>\alpha \}   } \mu(\tau,\lambda) d\lambda\right)^{\frac{1}{r}}\\
        &= \sup\limits_{\alpha>0}\,\alpha \left( \int\limits_{\{ \lambda \in \R : |m(\lambda)|>\alpha \}   } \mu(\tau,\lambda) d\lambda\right)^{\frac{1}{r}}.
    \end{align*}
    Therefore for $1<p\leq 2\leq q <\infty$, utilizing the above formula and taking the power $\frac{1}{r}$ inside the sum in \eqref{ineq_Fl,nT}, we obtain that
    \begin{align*}
        \|T_m f \|_{L^q(G)_{l,n}} \lesssim \left( \sup\limits_{\alpha>0}\,\alpha \left( \int\limits_{\{ \lambda \in \R : |m(\lambda)|>\alpha \}   } \mu(\tau,\lambda) d\lambda\right)^{\frac{1}{p}-\frac{1}{q}}+ \left(\sum_{k \in i\Gamma_{l,n}}   |m(k)||k|^{\frac{1}{p}-\frac{1}{q}} \right)\right) \|f\|_{L^p(G)_{l,n}}
    \end{align*}
    for all $f \in L^p(G)_{l,n}$, proving the theorem.
\end{proof}
\begin{theorem}\label{thm_spec_mult}    Let $1<p\leq 2\leq q<\infty$. Suppose $\varphi: \R \rightarrow \C$ is a function such that $|\varphi|$ decreases monotonically and continuously on $[\frac{1}{4}, \infty)$, with $\lim_{s\rightarrow \infty} |\varphi(s)| = 0$. Then the spectral multiplier operator $\varphi(\Omega)$  defined by
   \begin{align}\label{def_spec_m}
 	\varphi(\Omega) f(x) &= \frac{1}{4\pi^2} \int_{\R} \varphi\left(\frac{1+\lambda^2}{4}\right) \what {f_H}(\lambda) \phi_{\tau,\lambda}^{l,n} (x) \mu(\tau,\lambda) d\lambda 
 	+\frac{1}{2\pi}\sum_{k\in i\Gamma_{l,n}} {\varphi\left(\frac{1+k^2}{4}\right)} \what{ f_B}(k) \psi_{k}^{l,n}(x) |k|
 \end{align}
   satisfies the following estimates
   \begin{enumerate}
       \item\label{est_op_+} For $\tau = \tau_+$, we have 
   \begin{align*}
       \|\varphi(\Omega)\|_{L^p(G)_{l,n} \rightarrow L^q(G)_{l,n}  } \lesssim \left( \sum_{k\in i\Gamma_{l,n}} \left| \varphi\left(\frac{1+k^2}{4}\right) \right| |k|^{\frac{1}{p}-\frac{1}{q}}  \right) + \max \begin{cases}
           \sup\limits_{\frac{1}{4} < s \leq  \frac{1}{2}} |\varphi(s)| (s -\frac{1}{4})^{\frac{3}{2} (\frac{1}{p}- \frac{1}{q})},\\
            \sup\limits_{s \geq  \frac{1}{2}} |\varphi(s)| (s -\frac{1}{4})^{ (\frac{1}{p}- \frac{1}{q})}.
       \end{cases} 
   \end{align*}
   where we recall that $ \|\varphi(\Omega)\|_{L^p(G)_{l,n} \rightarrow L^q(G)_{l,n}  }$ denotes the operator norm of $\varphi(\Omega)$ from $L^p(G)_{l,n}$ to $L^q (G)_{l,n}$.
   \item\label{est_op_-}  When $\tau=\tau_-$, we have
   \begin{align*}
       \|\varphi(\Omega)\|_{L^p(G)_{l,n} \rightarrow L^q(G)_{l,n}  } \lesssim \left( \sum_{k\in i\Gamma_{l,n}} \left| \varphi\left(\frac{1+k^2}{4}\right) \right| |k|^{\frac{1}{p}-\frac{1}{q}} \right)   + \max \begin{cases}
           \sup\limits_{\frac{1}{4} < s \leq  \frac{1}{2}} |\varphi(s)| (s -\frac{1}{4})^{\frac{1}{2} (\frac{1}{p}- \frac{1}{q})},\\
            \sup\limits_{s \geq  \frac{1}{2}} |\varphi(s)| (s -\frac{1}{4})^{ (\frac{1}{p}- \frac{1}{q})}.
       \end{cases}
   \end{align*}
   \end{enumerate}
\end{theorem}
\begin{proof}
    Comparing the definition of $\varphi(\Omega)$ with \eqref{mult_(l,n)},  we observe that $\varphi(\Omega)$ is a Fourier multiplier operator $T_m$ with symbol
    $$m(\zeta)= \varphi \left(\frac{1+\zeta^2}{4}\right) \quad \text{for all } \zeta \in \R \cup i\Gamma_{l,n}.$$
    Thus, applying Theorem \ref{lp-lq-mutli}, we get
    \begin{equation}
    \begin{aligned}\label{est_varom}
        \|\varphi(\Omega)\|_{L^p(G)_{l,n} \rightarrow L^q(G)_{l,n}}&\\ &\hspace{-2cm}\lesssim \left( \sum_{k \in i\Gamma_{l,n}}   \left|\varphi \left(\frac{1+k^2} {4} \right)\right||k|^{\frac{1}{p}-\frac{1}{q}} \right)+ \sup_{\alpha>0}  \alpha \left( \int\limits_{\{  \lambda \in \R : \left| \varphi\left(\frac{1+\lambda^2}{4}\right) \right|> \alpha\}} \mu(\tau,\lambda) d\lambda \right)^{\frac{1}{p}-\frac{1}{q}}. 
    \end{aligned}
    \end{equation}
Given that $\varphi$ is continuously decreasing on the interval $[\frac{1}{4}, \infty)$, it follows that
    \begin{align*}
        \sup_{\alpha>0}  \alpha \left( \int\limits_{\{  \lambda \in \R : \left|\varphi\left(\frac{1+\lambda^2}{4}\right) \right|> \alpha\}} \mu(\tau,\lambda) d\lambda \right)^{\frac{1}{p}-\frac{1}{q}}=  \sup_{0<\alpha < \left|\varphi(\frac{1}{4})\right|}  \alpha \left( \int\limits_{\{  \lambda \in \R : \left|\varphi\left(\frac{1+\lambda^2}{4}\right)\right|> \alpha\}} \mu(\tau,\lambda) d\lambda \right)^{\frac{1}{p}-\frac{1}{q}}.
    \end{align*}   
     For $\alpha \in (0,|\varphi(\frac{1}{4})|)$, using the hypotheses on $\varphi$, we can express $\alpha$ as $\alpha =|\varphi(s)|$ for a certain $s \in (\frac{1}{4},\infty)$. Then, utilizing the monotonicity of $|\varphi|$, we derive
    \begin{align*}
         \sup_{\alpha>0}  \alpha \left( \int\limits_{\{  \lambda \in \R : \left| \varphi\left(\frac{1+\lambda^2}{4}\right) \right|> \alpha\}} \mu(\tau,\lambda) d\lambda \right)^{\frac{1}{p}-\frac{1}{q}} &= \sup_{\left|\varphi(s)\right|< \left|\varphi(\frac{1}{4})\right|}  \left|\varphi(s) \right| 
         \left( \int\limits_{\{  \lambda \in \R : \left| \varphi\left(\frac{1+\lambda^2}{4}\right) \right|> \left| \varphi(s) \right| \}} \mu(\tau,\lambda) d\lambda \right)^{\frac{1}{p}-\frac{1}{q}}\\
         &= \sup_{s\geq  \frac{1}{4}}  |\varphi(s)| \left( \int\limits_{\{  \lambda \in \R :|\lambda| <  2 \sqrt{s-\frac{1}{4}} \}} \mu(\tau,\lambda) d\lambda \right)^{\frac{1}{p}-\frac{1}{q}}.
    \end{align*}    
Considering the expression \eqref{defn_mu_t} of $\mu(\tau,\lambda)$ and utilizing  the estimates of $\tanh \lambda$ and $\coth \lambda$ for $\lambda$ near and away from zero, we can write the following
\begin{align*}
    |\mu(\tau,\lambda)| \asymp\begin{cases}
        \lambda^2 (1+|\lambda|)^{-1} \quad &\text{if } \tau=\tau_+\\
        (1+|\lambda|) \quad &\text{if } \tau=\tau_-.
    \end{cases}
\end{align*}
Utilizing the estimate above for $\tau=\tau_+$, we obtain
    \begin{align}
          \sup_{\alpha>0}  \alpha \left( \int\limits_{\{  \lambda \in \R : \left| \varphi\left(\frac{1+\lambda^2}{4}\right)  \right|> \alpha\}} \mu(\tau_+,\lambda) d\lambda \right)^{\frac{1}{p}-\frac{1}{q}} & \lesssim \max  \begin{cases}
              \sup\limits_{\frac{1}{4} < s \leq \frac{1}{2}}  \left|\varphi(s)  \right| \left( \int_0^{2 \sqrt{s-\frac{1}{4}}} \lambda^2 d\lambda \right)^{\frac{1}{p}-\frac{1}{q}},\\
              \sup\limits_{ s \geq \frac{1}{2}}  \left| \varphi(s) \right|  \left( \int_0^{1} \lambda^2 d\lambda +\int_1^{{2 \sqrt{s-\frac{1}{4}}}} \lambda d\lambda\right)^{\frac{1}{p}-\frac{1}{q}}.
          \end{cases} \nonumber \\
          & \lesssim \max \begin{cases} \sup\limits_{\frac{1}{4} < s\leq \frac{1}{2}}  \left| \varphi(s) \right|
              \left( s- \frac{1}{4}\right)^{\frac{3}{2}(\frac{1}{p}-\frac{1}{q})},\\
              \sup\limits_{ s\geq \frac{1}{2}}  \left| \varphi(s) \right| \left( s- \frac{1}{4}\right)^{(\frac{1}{p}-\frac{1}{q})}.
          \end{cases}\label{est_t-}
    \end{align}
    For $\tau = \tau_-$, we note that $|\mu(\tau_-,\lambda)| \asymp 1$ when $\lambda$ is near zero. Thus, by utilizing this observation and performing a calculation similar to the one above, we obtain in this case
    \begin{align}\label{est_t+}
           \sup_{\alpha>0}  \alpha \left( \int\limits_{\{  \lambda \in \R : \left| \varphi\left(\frac{1+\lambda^2}{4}\right) \right|> \alpha\}} \mu(\tau_-,\lambda) d\lambda \right)^{\frac{1}{p}-\frac{1}{q}}  \lesssim \max \begin{cases} \sup\limits_{\frac{1}{4} < s\leq \frac{1}{2}}  \left| \varphi(s) \right|
              \left( s- \frac{1}{4}\right)^{\frac{1}{2}(\frac{1}{p}-\frac{1}{q})},\\
              \sup\limits_{ s\geq \frac{1}{2}}  \left| \varphi(s) \right| \left( s- \frac{1}{4}\right)^{(\frac{1}{p}-\frac{1}{q})}.
          \end{cases}  
    \end{align}
   Finally, by substituting the estimates from \eqref{est_t+} and \eqref{est_t-} into \eqref{est_varom}, we conclude with the proof of the theorem.
\end{proof}
\begin{remark}
   We would like to note that if $\varphi, \varphi_1: \R \rightarrow \R$ are two functions such that $|\varphi_1| \leq |\varphi|$ and $\varphi$ satisfies all the hypotheses of Theorem \ref{thm_spec_mult}, then it follows that $ \|\varphi_1(\Omega)\|_{L^p(G)_{l,n} \rightarrow L^q(G)_{l,n}  }$ is dominated by the right hand side of the estimates in \eqref{est_op_+}, \eqref{est_op_-} in Theorem \ref{thm_spec_mult}.
\end{remark}
\section{Applications: Estimates for heat propagator of the Casimir element and  the global well-posedness of nonlinear equations}\label{sec_application}
This section is dedicated to the applications of our main results Theorem \ref{lp-lq-mutli} and \ref{thm_spec_mult} regarding the boundedness of the Fourier and spherical multipliers in $L^p$-$L^q$ spaces. We will begin with $L^p$-$L^q$ estimates of the heat propagator of the Casimir element for $(l,n)$ type functions on $G$.
\subsection{$L^p$-$L^q$ estimates of heat propagator of the Casimir element}\label{heat_l,n}
For fixed $ l,n \in \Z^{\tau}$, let $\Omega_{l,n}$ be the restriction of the differential operator $\Omega$ to $L^2(G)_{l,n}$. Since $\Omega$ preserves the types of a function and $l,n$ is always fixed,  we will use $\Omega$ instead of $\Omega_{l,n}$ for simplicity when $\Omega$ acts on $l,n$ type functions. Now, let us consider the following heat equation for $ 1<p\leq 2$ 
\begin{equation}
\begin{aligned}\label{23vis}
     u_t& =\Omega u,  \qquad t>0,\quad \text{and} \quad
    u(0,\cdot)&=u_0 \in L^p(G)_{l,n}. 
\end{aligned}
\end{equation}
 Then one can verify that for $t>0,$ 
 \begin{align*}
     u(t, x)=e^{t \Omega}u_0
 \end{align*} is the solution to the initial value problem \eqref{23vis}. Consequently, it represents the spectral multiplier \eqref{def_spec_m} operator with the symbol
 \begin{align*}
     \hspace{2cm} \varphi\left(\frac{1+\zeta^{2}}{4}\right)= e^{-t\left({\frac{1+\zeta^2}{4}}\right)},\qquad \text{for all } \zeta \in \R\cup i\Gamma_{l,n}.
 \end{align*}
Next, to establish the $L^p(G)_{l,n}$-$L^q(G)_{l,n}$ estimates for $q\in[2,\infty)$, we will utilize Theorem \ref{thm_spec_mult}. We observe that the function $s \mapsto e^{-t|s|}$ on $\R$ satisfies the hypotheses of Theorem \ref{thm_spec_mult}. Thus, applying Theorem \ref{thm_spec_mult}  with $\tau = \tau_+$, we obtain for all $l,n \in \mathbb{Z}^{\tau_+} $
\begin{align*}
     \|e^{t\Omega}\|_{L^p(G)_{l,n} \rightarrow L^q(G)_{l,n} } \lesssim \left( \sum_{k\in i\Gamma_{l,n}} \left| e^{-t\left(\frac{1+k^2}{4}\right)}\right| |k|^{\frac{1}{p}-\frac{1}{q}}  \right)+ \max \begin{cases} \sup\limits_{\frac{1}{4} < s\leq \frac{1}{2}}  e^{-ts}
              \left( s- \frac{1}{4}\right)^{\frac{3}{2}(\frac{1}{p}-\frac{1}{q})},\\
              \sup\limits_{ s\geq \frac{1}{2}}  e^{-ts} \left( s- \frac{1}{4}\right)^{(\frac{1}{p}-\frac{1}{q})}. 
              \end{cases} 
\end{align*}
By setting $x=(s-\frac{1}{4})^{\frac{1}{2}}$, the inequality above reduces to
\begin{equation}\label{Lp-q_heat_est}
\begin{aligned}
     \|e^{t\Omega}\|_{L^p(G)_{l,n} \rightarrow L^q(G)_{l,n} } \lesssim \left( \sum_{k\in i\Gamma_{l,n}} \left| e^{-t\left(\frac{1+k^2}{4}\right)} \right| |k|^{\frac{1}{p}-\frac{1}{q}} \right)  +  e^{-\frac{t}{4}} \max  \begin{cases}
             \sup\limits_{0< x  \leq \frac{1}{2}}  e^{-t x^2} x^{3(\frac{1}{p}-\frac{1}{q})},\\
            \sup\limits_{x \geq \frac{1}{2}}  e^{-t x^2} x^{2(\frac{1}{p}-\frac{1}{q})} .
          \end{cases} 
\end{aligned}
\end{equation}
To estimate the last term of the expression above, let us consider the following function 
\begin{align*}
    \psi_{\alpha,t}(x)= e^{-tx^2} x^{\frac{\alpha}{r}}, \quad \text{for }x > 0, 
\end{align*}
where $\alpha>0$ and $1/r=1/p-1/q$. By taking its derivative
\begin{align*}
    \frac{d}{dx} \psi_{\alpha,t} (x)= e^{-tx^2} \left( -2x^2 t+ \frac{\alpha}{r} \right) x^{\frac{\alpha}{r}-1} 
\end{align*}
we note that the function $ \frac{d}{dx} \psi_{\alpha,t}(x)$ can have zero only at $x_{\alpha,t} := \sqrt{\frac{\alpha}{2rt}}$. Moreover, the function changes sign from positive to negative at $x_{\alpha,t}$. Thus $x_{\alpha,t}$ is a point of maximum of $\psi_{\alpha,t}$ and so
\begin{align}\label{est_t>p,q}
   \sup_{  x >0}  e^{-t x^2}  x^{\frac{\alpha}{r}} =e^{-\frac{\alpha}{2r}}\left( \frac{\alpha}{2rt} \right)^{\frac{\alpha}{2r}} = C_{r,\alpha} \, t^{-\frac{\alpha}{2r}}. 
\end{align}
Here we observe that for a fixed $\alpha>0$, when $t$ goes towards zero, then $\sqrt{\frac{\alpha}{2rt}} $ goes towards infinity. Thus by taking $\alpha =2$ in \eqref{est_t>p,q}, we can write for $t\in (0,1]$
\begin{align}\label{t<1,h_t}
             \sup\limits_{0<x  \leq \frac{1}{2}}  e^{-t x^2} x^{3(\frac{1}{p}-\frac{1}{q})}\leq C \quad \text{and} \quad
            \sup\limits_{x \geq \frac{1}{2}}  e^{-t x^2} x^{2(\frac{1}{p}-\frac{1}{q})} \leq C t^{-(\frac{1}{p}-\frac{1}{q})}.
\end{align}
Next, for large $t$, we note $x\mapsto \psi_{2,t}(x) $ is a decreasing function on the interval $[\frac{1}{2},\infty)$. Sepcifically, for $t\geq \frac{4}{r}$, by utilizing the decreasing property of  $\psi_{2,t}(x)$, we derive
\begin{align}\label{t>1,x>1/2}
 \sup\limits_{x \geq \frac{1}{2}}  e^{-t x^2} x^{2(\frac{1}{p}-\frac{1}{q})} \leq  e^{- \frac{t}{4}} {2}^{-2(\frac{1}{p}-\frac{1}{q})}.
 \end{align}
 On the other hand, for large $t$, $\sqrt{\frac{\alpha}{2rt}} $ goes towards zero, so we can write from \eqref{est_t>p,q}
 \begin{align}\label{t>1,x<1/2}
      \sup\limits_{0< x  \leq \frac{1}{2}}  e^{-t x^2} x^{3(\frac{1}{p}-\frac{1}{q})} \leq C_{p,q}  t^{-\frac{3}{2}{(\frac{1}{p}-\frac{1}{q})}}.
 \end{align}
Again for large  $t$, we have $e^{-\frac{t}{4}}\lesssim t^{{-\frac{3}{2}{(\frac{1}{p}-\frac{1}{q})}}}$, and for $t$ within any compact interval on $(0,\infty)$,  the operator  norm $ \|e^{t\Omega}\|_{L^p(G)_{l,n} \rightarrow L^q(G)_{l,n} } $ can be bounded by a constant. Thus, substituting the estimates \eqref{t<1,h_t}, \eqref{t>1,x>1/2}, and \eqref{t>1,x<1/2} in \eqref{Lp-q_heat_est} yields the following
\begin{align*}
     \|e^{t\Omega}\|_{L^p(G)_{l,n} \rightarrow L^q(G)_{l,n} } \lesssim  \left( \sum_{k\in i\Gamma_{l,n}} \left| e^{-t\left(\frac{1+k^2}{4}\right)} \right| |k|^{\frac{1}{p}-\frac{1}{q}} \right) + \begin{cases}
         t^{-\left( \frac{1}{p}-\frac{1}{q}\right)}, \quad & \text{if } 0<t\leq 1,\\
         e^{-\frac{t}{4}} t^{-\frac{3}{2}\left( \frac{1}{p}-\frac{1}{q}\right)} & \text{if } t\geq 1,
     \end{cases}
\end{align*}
whence we obtain
\begin{align*}
     \|u(t, \cdot)\|_{L^q(G)_{l,n}} &= \|e^{-t \Omega}u_0\|_{L^q(G)_{l,n}} \\
     & \leq \|e^{-t \Omega}\|_{L^p(G)_{l,n} \rightarrow L^q(G)_{l,n} } \|u_0\|_{L^p(G)_{l,n}}\nonumber \\&\lesssim  \|u_0\|_{L^p(G)_{l,n}} \Bigg[ \left( \sum_{k\in i\Gamma_{l,n}} \left| e^{-t\left(\frac{1+k^2}{4}\right)} \right| |k|^{\frac{1}{p}-\frac{1}{q}} \right) \\& \hspace{4cm} \quad\quad+\begin{cases}
         t^{-\left( \frac{1}{p}-\frac{1}{q}\right)}, \quad & \text{if } 0<t\leq 1,\\
         e^{-\frac{t}{4}} t^{-\frac{3}{2}\left( \frac{1}{p}-\frac{1}{q}\right)} & \text{if } t\geq 1.
         \end{cases}\Bigg]
\end{align*}
Next, for the case where $\tau =\tau_-$, we can derive the following from Theorem \ref{thm_spec_mult} (2), after performing the change of variable $x= (s-\frac{1}{4})^{\frac{1}{2}}$
\begin{align}\label{Lp-q_heat_est_t-}
     \|e^{t\Omega}\|_{L^p(G)_{l,n} \rightarrow L^q(G)_{l,n} } \lesssim \left( \sum_{k\in i\Gamma_{l,n}} \left| e^{-t\left(\frac{1+k^2}{4}\right)} \right| |k|^{\frac{1}{p}-\frac{1}{q}} \right)+  e^{-\frac{t}{4}} \max\begin{cases} \sup\limits_{0< x \leq   \frac{1}{2}}  e^{-t x^2}
              x^{(\frac{1}{p}-\frac{1}{q})},\\
             \sup\limits_{x \geq \frac{1}{2}}  e^{-t x^2} x^{2(\frac{1}{p}-\frac{1}{q})} .
          \end{cases} 
\end{align}
Following a similar calculation as in the previous case, we can obtain
\begin{align*}
     \|u(t, \cdot)\|_{L^q(G)_{l,n}} &= \|e^{-t \Omega}u_0\|_{L^q(G)_{l,n}} \\
     & \leq \|e^{-t \Omega}\|_{L^p(G)_{l,n} \rightarrow L^q(G)_{l,n} } \|u_0\|_{L^p(G)_{l,n}}\nonumber \\&\lesssim  \|u_0\|_{L^p(G)_{l,n}} \Bigg[ \left( \sum_{k\in i\Gamma_{l,n}} \left| e^{-t\left(\frac{1+k^2}{4}\right)} \right| |k|^{\frac{1}{p}-\frac{1}{q}} \right) \\& \hspace{4cm} +\begin{cases} t^{-(\frac{1}{p}-\frac{1}{q})} \quad & \text{if}\quad 0<t \leq 1\\e^{-\frac{t}{4}}\,\,  t^{- \frac{3}{2}(\frac{1}{p}-\frac{1}{q})} \quad & \text{if}\quad t \geq  1.\end{cases} \Bigg]
\end{align*}
\begin{remark}\label{rem_dec_heat}
If either $l$ or $n$ equals zero, the definition \eqref{def_Gam_l,n} of $\Gamma_{l,n}$ renders it a null set. Consequently, in these cases, we observe that
    \begin{align*}
         \|u(t, \cdot)\|_{L^q(G)_{l,n}}\lesssim   \|u_0\|_{L^p(G)_{l,n}} \left( \begin{cases} t^{-(\frac{1}{p}-\frac{1}{q})} \quad & \text{if}\quad 0<t \leq 1, \\e^{-\frac{t}{4}}\,\,  t^{- \frac{3}{2}(\frac{1}{p}-\frac{1}{q})} \quad & \text{if}\quad t \geq 1, \end{cases}\right)
    \end{align*}
    showing decay as $t$ approaches infinity, similar to what occurs in symmetric spaces \cite[Theorem 3.2]{CGM93} in the $K$-biinvariant setting. However, in general cases, we observe growth in the solution of the heat equation near infinity. As mentioned in the introduction, this growth arises from an additional discrete component of the Plancherel measure in $G$. Consequently, in the full group scenario, we believe the discrete part will contribute to the solution's growth as time progresses towards infinity in the $\Omega$-heat equation.
\end{remark}
We next utilize Theorem \ref{thm_spec_mult} to obtain Sobolev type embeddings for the Casimir element.
\subsection{Sobolev type embeddings} To obtain the Sobolev type inequality for the Casimir element $\Omega$ defined by \eqref{Casimirele}, we apply Theorem \ref{thm_spec_mult}  for the spectral multiplier defined by the symbol 
\begin{align*}
     \hspace{2cm} \varphi\left(\frac{1+\zeta^{2}}{4}\right)= \left(1+ \left({\frac{1+\zeta^2}{4}}\right) \right)^{b-a},\qquad \text{for all } \zeta \in \R\cup i\Gamma_{l,n},
 \end{align*}
 for $a>b.$ We note that for $a>b$,  the function  $s\mapsto \left(1+s\right)^{-(a-b)}$, $s\geq 0$, decreases monotonically and continuously on $[\frac{1}{4}, \infty)$, and approaches to zero at infinity.  Therefore,  as an application of Theorem \ref{thm_spec_mult} $(1)$ for $\tau=\tau_+$, we deduce that for $u \in C_c^{\infty}(G)_{l,n}$

\begin{align*}
    \|(1-\Omega)^{-(a-b)} u\|_{L^q(G)_{l,n}  }& \lesssim \|u\|_{L^p(G)_{l,n}} \Bigg[ \left( \sum_{k\in i\Gamma_{l,n}} \left| \left(1+ \left({\frac{1+k^2}{4}}\right)  \right)\right|^{-(a-b)} \right)^{\frac{1}{p}-\frac{1}{q}}  \\&\qquad+  \max \begin{cases} \sup\limits_{\frac{1}{4} \leq s \leq  \frac{1}{2}} (1+s)^{-(a-b)} (s -\frac{1}{4})^{\frac{3}{2} (\frac{1}{p}- \frac{1}{q})},\\
      \sup\limits_{s \geq  \frac{1}{2}} (1+s)^{-(a-b)} (s -\frac{1}{4})^{ (\frac{1}{p}- \frac{1}{q})}. \end{cases} \Bigg] 
\end{align*}
By considering the supremum near $s=\infty$, we observe that the following expression 
\begin{align*}
    \left( \sum_{k\in i\Gamma_{l,n}} \left| \left(1+ \left({\frac{1+k^2}{4}}\right)  \right)\right|^{-(a-b)} \right)^{\frac{1}{p}-\frac{1}{q}}  + \max \begin{cases} \sup\limits_{\frac{1}{4} \leq s \leq  \frac{1}{2}} (1+s)^{-(a-b)} (s -\frac{1}{4})^{\frac{3}{2} (\frac{1}{p}- \frac{1}{q})},\\
      \sup\limits_{s \geq  \frac{1}{2}} (1+s)^{-(a-b)} (s -\frac{1}{4})^{ (\frac{1}{p}- \frac{1}{q})}. \end{cases} 
       \end{align*}
is finite provided that $$(a-b)\geq \frac{1}{p}-\frac{1}{q}.$$
Therefore, for $1<p \leq 2 \leq  q <\infty $ and $(\frac{1}{p}-\frac{1}{q}) < a-b$ we have that 
\begin{equation} \label{sobem}
    \|(1-\Omega)^{b} u\|_{L^q(G)_{l,n}  } \lesssim  \|(1-\Omega)^{a} u\|_{L^p(G)_{l,n}  }.
\end{equation}
In a similar way, for $\tau=\tau_{-},$ as an application of Theorem \ref{thm_spec_mult} $(2)$ we get the same conclusion \eqref{sobem}. In particular, by taking $b=0$ and defining Sobolev space $W^{\alpha, p}(G)_{l, n}, \alpha>0,$ by the norm $\|u\|_{W^{\alpha, p}(G)_{l, n}}:=\|(1-\Omega)^{\frac{\alpha}{2}} u\|_{L^p(G)_{l,n} }$ we get the following Sobolev type embedding
\begin{equation}
     \| u\|_{L^q(G)_{l,n}  } \lesssim  \|(1-\Omega)^{\frac{\alpha}{2}} u\|_{L^p(G)_{l,n}  }=\|u\|_{W^{\alpha, p}(G)_{l, n}}
\end{equation}
provided that $\alpha\geq  \left(\frac{1}{p}-\frac{1}{q} \right)$ with $1<p\leq 2 \leq q<\infty.$

\subsection{Global well-posedness of nonlinear equations}  This subsection is dedicated to the implications of our main result Theorem \ref{lp-lq-mutli}. Our approach follows the general outline of \cite{DKRN} concerning Fourier analysis linked with the biorthogonal eigenfunction expansion of a model operator on smooth manifolds characterized by a discrete spectrum.  Now, we will delve into the well-posedness of nonlinear abstract Cauchy problems for $(l,n)$-type functions on $G$.
\subsubsection{Nonlinear heat equations}
 We will now consider the following Cauchy problem for a nonlinear evolution equation in the space $L^\infty(0, T, L^2(G)_{l,n}),$
 \begin{equation} \label{heat}
     u_t-|Bu(t)|^p=0,\quad u(0)=u_0,
 \end{equation}
 where $1<p<\infty$ and $B$ is a Fourier multiplier defined in \eqref{mult_(l,n)} on $L^2(G)_{l,n}$.

We say that the heat equation \eqref{heat} has a solution $u$ if the following integral equation holds:
\begin{equation}\label{heatsol}
u(t)=u_0+\int_0^t |Bu(\tau)|^p d\tau
\end{equation}
in the space $L^\infty(0, T, L^p(G)_{l,n})$ for every $T<\infty$. Additionally, we define $u$ as a local solution of \eqref{heat} if it satisfies equation \eqref{heatsol} in the space $L^\infty(0, T^*, L^2(G)_{l,n})$ for some $T^*>0$.
 Then, we have the following result.
 \begin{theorem}
 Let $1<p<\infty.$ Assume that $B$ is a Fourier multiplier defined in \eqref{mult_(l,n)},  such that its symbol $m$ satisfies $$\sup_{\alpha>0} \alpha \left( \int\limits_{\{  \lambda \in \R : |m(\lambda)|> \alpha\}} \mu(\tau,\lambda) d\lambda \right)^{\frac{1}{2}-\frac{1}{2p}} + \sum_{k \in i\Gamma_{l,n}}   |m(k)| |k|^{\frac{1}{2}-\frac{1}{2p}}<\infty.$$ Then the Cauchy problem \eqref{heat} has a local solution in the space $L^\infty(0, T^*, L^2(G)_{l,n})$ for some $T^*>0.$ 
 \end{theorem}
 \begin{proof}
 By integrating equation \eqref{heat} with respect to $t$, we obtain
 $$
u(t)=u_{0} + \int\limits_0^t |Bu(\tau)|^{p} d\tau.
$$
Taking the $L^2$-norm on both sides yields
\begin{equation}\label{nh_sol}
\begin{aligned}
    \|u(t)\|_{L^2(G)_{l,n}}^{2}  &\leq C \Bigg(\|u_0\|_{L^2(G)_{l,n}}^2+\left\| \int_0^t |Bu(t)|^p\, d\tau \right\|^2_{L^2(G)_{l,n}} \Bigg)\\& = C \Bigg(\|u_0\|_{L^2(G)_{l,n}}^2+ \int_{G} \left| \int_0^t |Bu(t)|^p\, d\tau  \right|^2 \,dx\Bigg).
\end{aligned}
\end{equation}
Applying the H\"older's inequality, we get
\begin{align*}
    \int_0^t|Bu(\tau)|^p\, d\tau \leq \left(\int_0^t 1 \, d\tau \right)^{\frac{1}{2}} \left(\int_0^t|Bu(\tau)|^{2p}\, d\tau \right)^{\frac{1}{2}}= t^{\frac{1}{2}} \left(\int_0^t\left|Bu(\tau) \right|^{2p}\, d\tau \right)^{\frac{1}{2}}.
\end{align*}
Substituting the inequality above in \eqref{nh_sol}, it follows that
\begin{align*}
    \|u(t)\|_{L^2(G)_{l,n}}^{2}  &\leq C \Bigg(\|u_0\|_{L^2(G)_{l,n}}^2+ t \int_{G}   \int_0^t |Bu(t)|^{2p}\, d\tau\,  \,dx\Bigg)\\& \leq  C \Bigg(\|u_0\|_{L^2(G)_{l,n}}^2+ t    \int_0^t \int_{G} |Bu(t)|^{2p}\,  \,dx\, d\tau\Bigg) \\& = C \Bigg(\|u_0\|_{L^2(G)_{l,n}}^2+ t    \int_0^t \|Bu(t)\|^{2p}_{L^{2p}(G)_{l,n}}\, d\tau\Bigg).
\end{align*}
Next, utilizing the hypothesis on the symbol of $B$, it follows from Theorem \ref{lp-lq-mutli} that $B$ is a bounded operator from $L^2(G)_{l,n}$ to $L^{2p}(G)_{l,n}$, moreover $\|B u(t)\|_{L^{2p}(G)_{l,n}} \leq C\|u(t)\|_{L^2(G)_{l,n}}$. This implies
\begin{align}\label{EQ:space-norm}
    \|u(t)\|_{L^2(G)_{l,n}}^{2} \leq C \Bigg(\|u_0\|_{L^2(G)_{l,n}}^2+ t    \int_0^t  \|u(t)\|^{2p}_{L^{2p}(G)_{l,n}}\, d\tau\Bigg),
\end{align}
for some constant $C$ independent from {$u_0$} and $t$. By taking $L^{\infty}$-norm in time on both sides of the estimate \eqref{EQ:space-norm}, we obtain
\begin{equation}\label{EQ:time-space-norm}
\|u\|_{L^{\infty}(0, T; L^2(G)_{l,n})}^{2}\leq C\Big(\|u_{0}\|_{L^2(G)_{l,n}}^{2} + T^{2} \|u\|^{2p}_{L^{\infty}(0, T; L^2(G)_{l,n})}\Big).
\end{equation}
Let us introduce the following set
\begin{equation}\label{u-Set}
\mathfrak{O}_c:=\left\{u\in L^{\infty}(0, T; L^2(G)_{l,n}): \|u\|_{L^{\infty}(0, T; L^2(G)_{l,n})} \leq c \|u_{0}\|_{L^2(G)_{l,n}}\right\},
\end{equation}
for some constant $c > 1$. Then, for $u \in \mathfrak{O}_c$ we have
$$
\|u_{0}\|_{L^2(G)_{l,n}}^{2} + T^{2} \|u\|^{2p}_{L^{\infty}(0, T; L^2(G)_{l,n})} \leq 
\|u_{0}\|_{L^2(G)_{l,n}}^{2} + T^{2} c^{2p} \|u_0\|^{2p}_{L^2(G)_{l,n}}.
$$
Finally, for $u$ to be from the set $\mathfrak{O}_c$ it is enough to have, by invoking \eqref{EQ:time-space-norm}, that 
$$
\|u_{0}\|_{L^2(G)_{l,n}}^{2} + T^{2} c^{2p} \|u_0\|^{2p}_{L^2(G)_{l,n}}\leq c^{2} \|u_0\|_{L^2(G)_{l,n}}^{2}.
$$
It can be obtained by requiring the following,
$$
T \leq T^{\ast}:=\frac{\sqrt{c^{2}-1}}{c^{p}\|u_0\|_{L^2(G)_{l,n}}}.
$$
Thus, by applying the fixed point theorem, there exists a unique local solution $u\in L^{\infty}(0, T^{\ast};$ $L^2(G)_{l,n})$ of the Cauchy problem \eqref{heat}.
 \end{proof}
 
\subsubsection{Nonlinear wave equation} In this subsection, we will consider that the initial value problem  for the nonlinear wave equation  
\begin{align}\label{E-WNLE}
u_{tt}(t) - \Psi(t)|Bu(t)|^{p} = 0,
\end{align} with the initial condition 
$$
u(0)=u_0, \,\,\, u_t(0)=u_1,
$$
where  $1< p<\infty$, $\Psi$ is a positive bounded function depending only on time, and $B$ is a linear operator initially defined in $L^2(G)_{l,n}$. We will study the well-posedness of the equation \eqref{E-WNLE}.

We say that the initial value problem \eqref{E-WNLE} admits a global solution $u$ in $[0,T]$ if it satisfies
\begin{equation}\label{E-WNLE-Sol}
u(t)=u_{0} + t u_{1} + \int\limits_0^t (t-\tau) \Psi(\tau) |Bu(\tau)|^{p} d\tau
\end{equation}
in the space $L^{\infty}(0, T; L^2(G)_{l,n})$ for every $T<\infty$.

We say that \eqref{E-WNLE} admits a local solution $u$ if it satisfies
the equation \eqref{E-WNLE-Sol} in the space $L^{\infty}(0, T^{\ast}; L^2(G)_{l,n})$ for some $T^{\ast}>0$.

Then, we have the following result regarding the global and local solutions of \eqref{E-WNLE}.
\begin{theorem}\label{Th: E-WNLE}
Let $1< p<\infty$. Suppose that $B$ is a Fourier multiplier defined in \eqref{mult_(l,n)}, such that its symbol $m$ satisfies $$\sup_{\alpha>0} \alpha \left( \int\limits_{\{  \lambda \in \R : |m(\lambda)|> \alpha\}} \mu(\tau,\lambda) d\lambda \right)^{\frac{1}{2}-\frac{1}{2p}} + \sum_{k \in i\Gamma_{l,n}}   |m(k)| |k|^{\frac{1}{2}-\frac{1}{2p}}<\infty.$$

\begin{itemize}
    \item [(i)] Assume that $\|\Psi\|_{L^{2}(0, T)}<\infty$ for some $T>0$. Then the Cauchy problem \eqref{E-WNLE} has a local solution in  $L^{\infty}(0, T; L^2(G)_{l,n})$.
    \item [(ii)] Suppose $u_1$ is identically equal to zero, and let $\gamma > 3/2$. Additionally, assume that there exists a constant $c>0$ such that $$\|\Psi\|_{L^{2}(0, T)} \leq c  T^{-\gamma}$$ for every $T > 0$. Then, for every $T > 0$, the Cauchy problem \eqref{E-WNLE} has a global solution in the space $L^{\infty}(0, T; L^2(G)_{l,n})$ for sufficiently small $u_0$ in the $L^2$-norm.
\end{itemize}
\end{theorem}
\begin{proof} (i) By integrating the equation \eqref{E-WNLE} two times  in $t$ one gets
$$
u(t)=u_{0} + t u_{1} + \int\limits_0^t (t-\tau) \Psi(\tau) |Bu(\tau)|^{p} d\tau.
$$
Taking the $L^2$-norm on both sides, for $t < T$, we obtain
\begin{align*}
    \|u(t)\|_{L^2(G)_{l,n}}^{2}&\leq  C\left\{ \|u_{0}\|_{L^2(G)_{l,n}}^{2} + t^2 \|u_{1}\|_{L^2(G)_{l,n}}^{2}+ \left\|\int\limits_0^t (t-\tau) \Psi(\tau) |Bu(\tau)|^{p} d\tau \right\|_{L^2(G)_{l,n}}^{2} \right\} \\& \leq C\left\{ \|u_{0}\|_{L^2(G)_{l,n}}^{2} + t^2 \|u_{1}\|_{L^2(G)_{l,n}}^{2}+ \int_{G} \Big|\int\limits_0^t (t-\tau) \Psi(\tau) |Bu(\tau)|^{p} d\tau \Big|^2  dx \right\} \\& \leq C\left\{ \|u_{0}\|_{L^2(G)_{l,n}}^{2} + t^2 \|u_{1}\|_{L^2(G)_{l,n}}^{2}+ \int_{G} \Big(t \int\limits_0^t \Big| \Psi(\tau)|Bu(\tau)|^{p}\Big| d\tau \Big)^{2}  dx \right\}.
\end{align*}
Now, applying the H\"older's inequality on the last integral, it follows that
\begin{align*}
    \|u(t)\|_{L^2(G)_{l,n}}^{2} & \leq C \left\{ \|u_{0}\|_{L^2(G)_{l,n}}^{2} + t^2 \|u_{1}\|_{L^2(G)_{l,n}}^{2}+ \int_{G} t^{2} \int\limits_0^t \Big| \Psi(\tau)\Big|^{2} d\tau \int\limits_0^t \Big| Bu(\tau)\Big|^{2p} d\tau  dx \right\} \\& \leq C\left\{ \|u_{0}\|_{L^2(G)_{l,n}}^{2} + t^2 \|u_{1}\|_{L^2(G)_{l,n}}^{2}+ t^2 \|\Psi\|_{L^2(0,T)}^2 \int_{G}  \int\limits_0^t \Big| Bu(\tau)\Big|^{2p} d\tau  dx \right\} \\& \leq C\left\{ \|u_{0}\|_{L^2(G)_{l,n}}^{2} + t^2 \|u_{1}\|_{L^2(G)_{l,n}}^{2}+ t^2 \|\Psi\|_{L^2(0,T)}^2 \int\limits_0^t \|Bu(\tau)\|^{2p}_{L^{2p}(G)_{l,n}} d\tau \right\}.
\end{align*}
Next, utilizing the hypothesis on the symbol of $B$, it follows from Theorem \ref{lp-lq-mutli} that $B$ is a bounded operator from $L^2(G)_{l,n}$ to $L^{2p}(G)_{l,n}$, moreover, $\|B u(t)\|_{L^{2p}(G)_{l,n}} \leq C\|u(t)\|_{L^2(G)_{l,n}}$. Therefore, the inequality above yields
\begin{equation}
\label{EQ: WE-space-norm}
\|u(t)\|_{L^2(G)_{l,n}}^{2}\leq C(\|u_{0}\|_{L^2(G)_{l,n}}^{2} + t^2 \|u_{1}\|_{L^2(G)_{l,n}}^{2}+ t^{2} \|\Psi\|_{L^{2}(0, T)}^{2} \int\limits_0^t \|u(\tau)\|^{2p}_{L^{2p}(G)_{l,n}} d\tau),
\end{equation}
for some constant $C$ independent of $u_0, u_1$ and $t$. Finally, by taking the $L^{\infty}$-norm in time on both sides of the estimate \eqref{EQ: WE-space-norm}, one obtains
\begin{equation}
\label{EQ: WE-time-space-norm}
\|u\|_{L^{\infty}(0, T; L^2(G)_{l,n})}^{2}\leq C (\|u_{0}\|_{L^2(G)_{l,n}}^{2} + T^2 \|u_{1}\|_{L^2(G)_{l,n}}^{2}+ T^{3} \|\Psi\|_{L^{2}(0, T)}^{2} \|u\|^{2p}_{L^{\infty}(0, T; L^2(G)_{l,n})}). 
\end{equation}

Let us introduce the set
\begin{equation}
Q_c:=\Big\{u\in L^{\infty}(0, T; L^2(G)_{l,n}): \|u\|_{L^{\infty}(0, T; L^2(G)_{l,n})}^{2} \leq
c(\|u_{0}\|_{L^2(G)_{l,n}}^{2} + T^2 \|u_{1}\|_{L^2(G)_{l,n}}^{2})\Big\}
\end{equation}
for some constant $c \geq 1$. Then, for $u \in  Q_c$ we have
\begin{align}\label{WE-Est}
&\|u_{0}\|_{L^2(G)_{l,n}}^{2} + T^2 \|u_{1}\|_{L^2(G)_{l,n}}^{2} + T^{3} \|\Psi\|_{L^{2}(0, T)}^{2} \|u\|^{2p}_{L^{\infty}(0, T; L^2(G)_{l,n})} \nonumber \\
&\leq \|u_{0}\|_{L^2(G)_{l,n}}^{2} + T^2 \|u_{1}\|_{L^2(G)_{l,n}}^{2} + T^{3} \|\Psi\|_{L^{2}(0, T)}^{2} c^{p}\Big(\|u_{0}\|_{L^2(G)_{l,n}}^{2} + T^2 \|u_{1}\|_{L^2(G)_{l,n}}^{2}\Big)^{p}. 
\end{align}

We observe that, to have $u$ from the set $Q_c$ it is enough to have, by invoking \eqref{EQ: WE-time-space-norm} and using \eqref{WE-Est}, that 
\begin{align*}
\begin{split}
\|u_{0}\|_{L^2(G)_{l,n}}^{2} + &T^2 \|u_{1}\|_{L^2(G)_{l,n}}^{2} + T^{3} \|\Psi\|_{L^{2}(0, T)}^{2} c^{p}\Big(\|u_{0}\|_{L^2(G)_{l,n}}^{2} + T^2 \|u_{1}\|_{L^2(G)_{l,n}}^{2}\Big)^{p}\\
&\leq c(\|u_{0}\|_{L^2(G)_{l,n}}^{2} + T^2 \|u_{1}\|_{L^2(G)_{l,n}}^{2}).
\end{split}
\end{align*}
It can be obtained by requiring the following
$$
T \leq T^{\ast}:=\min\left[\left(\frac{c-1}{\|\Psi\|_{L^{2}(0, T)}^{2}c^{p}\|u_0\|_{L^2(G)_{l,n}}^{2p-2}}\right)^{1/3}, \, \left(\frac{c-1}{\|\Psi\|_{L^{2}(0, T)}^{2}c^{p}\|u_1\|_{L^2(G)_{l,n}}^{2p-2}}\right)^{\frac{1}{3}}\right].
$$
Thus, by applying the fixed point theorem, there exists a unique local solution $u\in L^{\infty}(0, T^{\ast};$ $ L^2(G)_{l,n})$ of the Cauchy problem \eqref{E-WNLE}.

To establish part (ii), we replicate the reasoning from the proof of part (i) to obtain \eqref{EQ: WE-time-space-norm}. Now, considering the assumptions on $u_1  (\equiv 0)$ and $\Psi$, the inequality \eqref{EQ: WE-time-space-norm} yields
\begin{equation}
\label{EQ: WE-time-space-norm-2}
\|u\|_{L^{\infty}(0, T; L^2(G)_{l,n})}^{2}\leq C \Big(\|u_{0}\|_{L^2(G)_{l,n}}^{2} + T^{3-2\gamma}  \|u\|^{2p}_{L^{\infty}(0, T; L^2(G)_{l,n})}\Big). 
\end{equation}

For a fixed constant $c \geq 1$, let us introduce the following set
$$
R_c:=\Big\{u\in L^{\infty}(0, T; L^2(G)_{l,n}): \|u\|_{L^{\infty}(0, T; L^2(G)_{l,n})}^{2} \leq c T^{\gamma_{0}}\|u_{0}\|_{L^2(G)_{l,n}}^{2}\Big\},
$$
with $\gamma_{0}>0$ is to be defined later. Now, note that
\begin{align*}
\|u_{0}\|_{L^2(G)_{l,n}}^{2} + T^{3-2\gamma}  \|u\|^{2p}_{L^{\infty}(0, T; L^2(G)_{l,n})} 
\leq \|u_{0}\|_{L^2(G)_{l,n}}^{2} + T^{3-2\gamma+\gamma_{0}p} c^{p} \|u_{0}\|_{L^2(G)_{l,n}}^{2p}.    
\end{align*}

To guarantee $u\in R_c$, by invoking \eqref{EQ: WE-time-space-norm-2} we require that
\begin{align*}
\|u_{0}\|_{L^2(G)_{l,n}}^{2} + T^{3-2\gamma+\gamma_{0}p} c^{p} \|u_{0}\|_{L^2(G)_{l,n}}^{2p} \leq c T^{\gamma_{0}} \|u_{0}\|_{L^2(G)_{l,n}}^{2}.   
\end{align*}
Now by choosing $0<\gamma_0<\frac{2\gamma-3}{p}$ such that
$
\tilde{\gamma}:=3-2\gamma+\gamma_{0}p<0,
$ we obtain
$$
c^{p} \|u_{0}\|_{L^2(G)_{l,n}}^{2p-2} \leq c T^{-\tilde{\gamma}+\gamma_{0}}.
$$
From the last estimate, we conclude that for any $T>0$ there exists sufficiently small $\|u_{0}\|_{L^2(G)_{l,n}}$ such that IVP \eqref{E-WNLE} has a solution. It proves part (ii) of Theorem \ref{Th: E-WNLE} and thus concludes the theorem.
\end{proof}

\section*{Acknowledgments}
 The authors are supported by the FWO Odysseus 1 grant G.0H94.18N: Analysis and Partial Differential Equations, the Methusalem program of the Ghent University Special Research Fund (BOF) (Grant number 01M01021). VK and MR are supported by FWO Senior Research Grant G011522N. TR is also supported by BOF postdoctoral fellowship at Ghent University BOF24/PDO/025.   MR is also supported by EPSRC grant EP/V005529/7.
\bibliographystyle{alphaurl}
\bibliography{ReferencesIII.bib}

\end{document}